\documentclass{amsart}

\usepackage{amssymb,amsmath}
\usepackage{pinlabel}
\usepackage{enumerate}
\usepackage{color}

\def\red#1{{\textcolor{red}{#1}}}

\newtheorem{prop}{Proposition}
\newtheorem{thm}[prop]{Theorem}

\theoremstyle{definition}

\newtheorem*{ex}{Example}
\newtheorem*{exs}{Examples}
\newtheorem*{app}{Applications}
\newtheorem*{term}{Terminology}
\newtheorem*{defn}{Definition}
\newtheorem*{ack}{Acknowledgements}

\newcommand{\co}{\colon\thinspace}

\newcommand{\C}{{\mathbb{C}}}
\newcommand{\CP}{{\mathbb{CP}}}
\newcommand{\N}{{\mathbb{N}}}
\newcommand{\R}{{\mathbb{R}}}
\newcommand{\RP}{{\mathbb{RP}}}

\newcommand{\frakp}{{\mathfrak{p}}}

\newcommand{\bfx}{{\mathbf{x}}}

\newcommand{\Z}{{\mathbb{Z}}}
\newcommand{\rmd}{{\mathrm{d}}}
\newcommand{\rme}{{\mathrm{e}}}
\newcommand{\rmi}{{\mathrm{i}}}
\newcommand{\xist}{{\xi_{\mathrm{st}}}}

\DeclareMathOperator{\id}{\mathrm{id}}

\DeclareMathOperator{\Int}{\mathrm{Int}}

\begin{document}

\author{Hansj\"org Geiges}
\address{Mathematisches Institut, Universit\"at zu K\"oln,
Weyertal 86--90, 50931 K\"oln, Germany}
\email{geiges@math.uni-koeln.de}

\title{How to depict $5$-dimensional manifolds}

\date{}

\begin{abstract}
We usually think of $2$-dimensional manifolds as surfaces embedded
in Euclidean $3$-space. Since humans cannot visualise Euclidean
spaces of higher dimensions, it appears to be impossible to give
pictorial representations of higher-dimensional manifolds.
However, one can in fact encode the topology of a surface
in a $1$-dimensional picture. By analogy, one can draw
$2$-dimensional pictures of $3$-manifolds (Heegaard diagrams),
and $3$-dimensional pictures of $4$-manifolds (Kirby diagrams).
With the help of open books one can likewise represent at least some
$5$-manifolds by $3$-dimensional diagrams, and contact geometry can be
used to reduce these to drawings in the $2$-plane.

In this paper, I shall explain how to draw such pictures and how to use them
for answering topological and geometric questions. The work on
$5$-manifolds is joint with Fan Ding and Otto van Koert.
\end{abstract}

\subjclass[2010]{57R65, 57M25, 57R17}

\maketitle
\section{Introduction}
\label{section:intro}
A \emph{manifold} of dimension $n$ is a topological space $M$
that locally `looks like' Euclidean $n$-space~$\R^n$; more precisely,
any point in $M$ should have an open neighbourhood
homeomorphic to an open subset of~$\R^n$. Simple examples (for $n=2$)
are provided by surfaces in~$\R^3$, see Figure~\ref{figure:surfaces}. Not all
$2$-dimensional manifolds, however, can be visualised in
$3$-space, even if we restrict attention to compact manifolds.
Worse still, these pictures `use up' all three spatial dimensions
to which our brains are adapted by natural selection.

\begin{figure}[h]
\labellist
\small\hair 2pt
\endlabellist
\centering
\includegraphics[scale=.48]{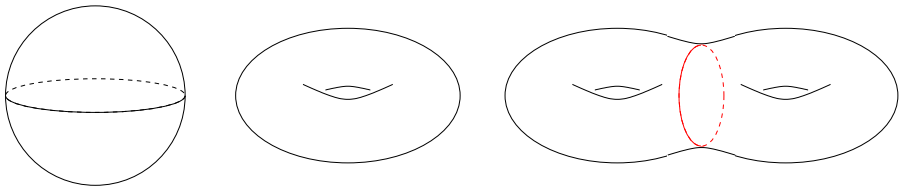}
  \caption{The $2$-sphere $S^2$, the $2$-torus $T^2$, and the surface
$\Sigma_2$ of genus two.}
  \label{figure:surfaces}
\end{figure}

So any attempt to visualise higher-dimensional manifolds
seems to be doomed. As regards $3$-dimensional manifolds, there
might be some hope to get an understanding from within, that is,
if we imagine ourselves travelling inside such a space.
This is not easy; even the critical Immanuel Kant seems to have taken
it as a given that a space which locally looks like $\R^3$ must needs
be $\R^3$ --- at least, that's how I interpret his dictum that
the space we live in is not amenable to empirical study:
``Space is not a conception which has been derived from outward
experiences. [...]
Space then is a necessary representation \emph{a priori}, which serves for
the foundation of all external intuitions. [...]
Space is represented as an infinite given quantity.''\footnote{``Der Raum
ist kein empirischer Begriff, der von \"au{\ss}eren Erfahrungen abgezogen
worden. [...] Der Raum ist eine notwendige Vorstellung, a priori, die
allen \"au{\ss}eren Anschauungen zum Grunde liegt. [...]
Der Raum wird als eine unendliche gegebene Gr\"o{\ss}e vorgestellt.''
\cite{kant56}, translation from~\cite{kant52}.}

Nonetheless, one can exercise one's
imagination. A good place to start is the beautiful book by
Weeks~\cite{week02}. See also \cite{pete79}, where it is
argued that Dante described an internal view of the
$3$-sphere in his \emph{Divina Commedia};
cf.~\cite{geig97,osse95}. A flight simulator for travel in
various $3$-dimensional manifolds can be found at\\
\texttt{http://www.geometrygames.org/CurvedSpaces/index.html.en}.\\
For more on the topology of the universe see the proceedings~\cite{tuc98},
which --- rather intriguingly --- contains a paper `Topology and the
universe' by Gott.

But how, then, is it possible to get a structural understanding of
higher-dimensio\-nal manifolds if we cannot visualise them
from without? The answer lies in a dimensional reduction
of the representation of surfaces. I shall describe how
to represent compact $2$-dimensional manifolds by
$1$-dimensional diagrams. To `represent' here means that the
diagram contains complete information about the global
topology of the surface. Moreover, there are rules for
manipulating such diagrams that enable us to prove which diagrams
represent homeomorphic surfaces.

Once this dimensional reduction has been grasped, one can proceed
by analogy up to $4$-dimensional manifolds, which should then
be representable by $3$-dimensional diagrams. Up to this
point, the material presented here is classical. The $1$-dimensional
approach to the classification of surfaces, which I have not found
discussed in detail elsewhere, has been tried and tested in
a lecture course on the geometry and topology of surfaces.

In Section~\ref{section:dim5} I shall present a
diagrammatic approach to the topology of $5$-manifolds
developed jointly with Fan Ding and Otto van Koert. The idea
here is to restrict attention to a class of $5$-manifolds
that can be described as special types of so-called
`open books'. These are decompositions of $5$-manifolds
into a collection of $4$-dimensional `pages' glued along
a $3$-dimensional `binding' that constitutes the
common boundary of these $4$-manifolds (where the
manifold looks like a closed half-space in~$\R^4$).
Under suitable assumptions, it suffices
to present a diagram of the $4$-dimensional page in order to
understand the topology of the $5$-manifold.
With a little help from contact geometry, as explained in the final section,
we can further simplify the $3$-dimensional
diagram of the page to a $2$-dimensional one.

\vspace{2mm}

This paper is an extended version of a colloquium talk I
have given at a number of universities. I have tried to keep the
colloquial style of the original presentation, but I have added
various technical details where it seemed appropriate for a
written account.
\subsection{Manifolds}
One usually postulates that a manifold $M$ should not only look
locally like~$\R^n$, but also that
\begin{enumerate}
\item[(M1)] $M$ is a topological Hausdorff space, i.e.\ any two
distinct points lie in disjoint neighbourhoods, and
\item[(M2)] the topology of $M$ has a countable base, i.e.\
there is a countable family of open subsets such that any
open subset of $M$ is a union of sets from this family.
\end{enumerate}

These requirements are largely a matter of technical convenience
and can safely be ignored for the purposes of this paper.
Condition (M1) prevents pathological examples such as a line with a double
point. Take $M$ to be a copy of the real line~$\R$, together
with an additional point $\ast$. Declare the topology on
$M=\R\cup\{\ast\}$ by $\R$ being an open subset of $M$, and
neighbourhoods of $\ast$ to be sets of the form $(U\setminus\{0\})
\cup\{\ast\}$, where $U\subset\R$ is a neighbourhood of the
origin $0\in\R$. This space $M$ is locally homeomorphic to~$\R$,
but not Hausdorff.

Condition (M2) will be redundant in this paper as we shall only be
concerned with compact manifolds. In general, one imposes this condition
to guarantee, for instance, metrisability of manifolds.
\begin{exs}
(1) The surface $\Sigma_2$
of genus two in Figure~\ref{figure:surfaces} can be obtained
from two copies of a $2$-torus $T^2$ as follows: remove the interior
of a small disc $D^2\subset T^2$ from each torus, then glue the resulting
circle boundaries $\partial (T^2\setminus\Int(D^2))=S^1$ by a
homeomorphism.

This operation is called a \emph{connected sum} and denoted by the
symbol $\#$. So $\Sigma_2$ is homeomorphic
to $T^2\# T^2$. Likewise, one can define the surface of genus $g\in\N$
as the connected sum of $g$ copies of $T^2$.

\vspace{1mm}

(2) There are surfaces that cannot so easily be visualised in $3$-space.
For instance, the real projective plane $\RP^2$ is defined as the
quotient space $S^2/\bfx\sim -\bfx$ obtained from the $2$-sphere
by identifying antipodal points. This space carries a natural topology,
the so-called \emph{quotient topology}, defined as follows. We have a
two-to-one projection map
\[ \begin{array}{rrcl}
\pi\co & S^2  & \longrightarrow & \RP^2\\
       & \bfx & \longmapsto     & [\bfx]
\end{array} \] 
sending each point in $S^2$ to its equivalence class in $\RP^2$.
A subset $U\subset\RP^2$ is called \emph{open} if its inverse
image $\pi^{-1}(U)$
is open in $S^2$ in the usual topology inherited from the inclusion
$S^2\subset\R^3$. This defines a topology on $\RP^2$; it is the
topology with the largest number of open sets for which the
projection $\pi$ is continuous.

The process of gluing, as in the connected sum construction,
has a similar description as a quotient. The resulting topology
is the obvious one you would expect from gluing two pieces of
paper along their edges.

The real projective plane is indeed a $2$-dimensional manifold:
an open neighbourhood of a point $\bfx\in S^2$ that does not contain
any pair of antipodal points descends homeomorphically
to a neighbourhood of $[\bfx] \in\RP^2$.

I claim that $\RP^2$ may be thought of as the gluing of a
$2$-disc and a M\"obius band along their circle boundary,
see Figure~\ref{figure:rp2}. Think of $S^2$ as being made up
of two polar caps, which are homeomorphic copies
of $D^2$, one being the antipodal image of the other.
In the quotient space $S^2/\bfx\sim-\bfx$ we can take one of these two
discs to represent the equivalence classes of points coming from
the polar caps. Likewise, the equivalence classes of points coming
from the band around the equator are represented by points in
one half of that band; when the ends of that strip are glued
with the antipodal map, this creates a M\"obius band. As we go once
around the boundary of a polar cap, we also go once along the
boundary of the M\"obius band.

\begin{figure}[h]
\labellist
\small\hair 2pt
\endlabellist
\centering
\includegraphics[scale=.45]{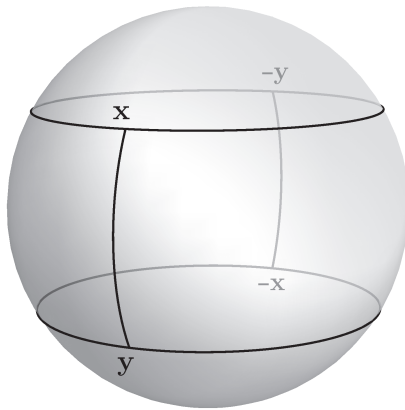}
  \caption{$\RP^2$ obtained by gluing a disc and a M\"obius band.}
  \label{figure:rp2}
\end{figure}

The real projective plane $\RP^2$ is an example of a
\emph{non-orientable} manifold: as one goes once along the
central circle in a M\"obius band, an oriented $2$-frame
returns with the opposite orientation. Moreover,
the real projective plane can be realised as a surface in $\R^3$
only at the price of allowing self-intersections. Such a realisation was
discovered by T.~Boy, see~\cite[\S\ 47]{hicv52}.
These two facts are correlated;
an elegant argument due to Samelson~\cite{same69} shows that
any compact hypersurface in Euclidean space is orientable.

\vspace{1mm}

(3) The next surface we can construct is the connected sum $\RP^2\#\RP^2$.
Since the removal of an open disc in $\RP^2$ leaves us with
a M\"obius band, this connected sum is the same as gluing
two M\"obius bands along the boundary. The resulting surface is
known as a Klein bottle. This explains why every Klein bottle bought at
\texttt{kleinbottle.com} comes with a product warning that, if dropped,
it may shatter into two M\"obius bands.

\vspace{1mm}

(4) In a first course on topology one learns that the connected sums
$\Sigma_g=\#_gT^2$, $g\in\N_0$ (the empty connected sum being~$S^2$, the
neutral element for the connected sum), and $\#_h\RP^2$, $h\in\N$,
constitute a complete list (without duplications) of the compact
surfaces. Where, you may ask, is the surface $T^2\#\RP^2$
or any other `mixed' connected sum in that list?
I shall give an answer to this question in Section~\ref{section:surfaces},
using the $1$-dimensional diagrammatic language developed there.

\vspace{1mm}

(5) A simple example of a compact $n$-dimensional manifold is the
$n$-sphere
\[ S^n:=\{\bfx\in\R^{n+1}\co \|\bfx\|=1\}.\]
Stereographic projection from any point $N$ of the sphere (regarded as
the north pole) onto the corresponding equatorial plane defines
a homeomorphism
\[ S^n\setminus\{N\}\stackrel{\cong}{\longrightarrow}\R^n.\]
\end{exs}
\subsection{Handle decompositions}
\label{subsection:handle}
We write $D^k$ for the closed $k$-dimensional unit disc (or ball),
\[ D^k:=\{\bfx\in\R^k\co \|\bfx\|\leq 1\}.\]
This is a $k$-dimensional manifold \emph{with boundary}; points $\bfx\in D^k$
with $\|\bfx\|=1$ have a neighbourhood in $D^k$ that looks
like an open subset in the half-space $\R^k_+=\{\bfx\in\R^k\co
x_k\geq 0\}$. The boundary is denoted
by~$\partial$; for instance, $\partial D^k=S^{k-1}$.

\begin{term}
In the literature, the term `manifold' is often used in the
wider sense so as to include manifolds with boundary. In this paper,
manifolds are always understood to be without boundary, unless
specified otherwise. Occasionally I shall use the standard term
\emph{closed manifold} for a \emph{compact} manifold
\emph{without} boundary when I wish to emphasise those attributes.
\end{term}

An $n$-dimensional $k$-handle is  a product $h_k:= D^k\times D^{n-k}$;
the number $k\in \{0,\ldots, n\}$ is called the \emph{index} of the handle.
I suppress $n$ from the notation for the handle, as the dimension will
be clear from the context. Up to homeomorphism, an $n$-dimensional
handle is of course simply a copy of the $n$-disc~$D^n$; what determines
the index of a handle is how it is used to build a manifold by
gluing handles.

A \emph{handle decomposition} of an $n$-dimensional manifold
$M^n$ is a way to write it as a union of handles, where we start
with a disjoint collection of $n$-discs, which are the zero handles
in the decomposition, and then successively attach handles, where a
handle of index $k$ is attached to the boundary of the
compound $X$ (or \emph{handlebody}) of all previous handles
along the part $\partial D^k\times D^{n-k}=S^{k-1}\times D^{n-k}$
of its boundary, see Figure~\ref{figure:k-handle}.
We write $\partial_- h_k$ for this part of
the boundary and call it the \emph{lower boundary} of the
$k$-handle. Similarly, we have the \emph{upper boundary}
$\partial_+h_k=D^k\times\partial D^{n-k}=D^k\times S^{n-k-1}$.
Figure~\ref{figure:k-handle} also shows the \emph{core disc}
$D^k\times\{0\}$ and the \emph{belt sphere} $\{0\}\times S^{n-k-1}$.

\begin{figure}[h]
\labellist
\small\hair 2pt
\pinlabel $h_k$ at 193 201
\pinlabel ${\partial X}$ at 634 204
\pinlabel $X$ at 550 38
\pinlabel ${\partial_-h_k=\partial D^k\times D^{n-k}}$ [t] at 360 30
\pinlabel ${\partial_+h_k=D^k\times S^{n-k-1}}$ [l] at 553 329
\pinlabel ${\{0\}\times S^{n-k-1}}$ [r] at 211 325
\pinlabel ${D^k\times\{0\}}$ [l] at 600 285
\endlabellist
\centering
\includegraphics[scale=.4]{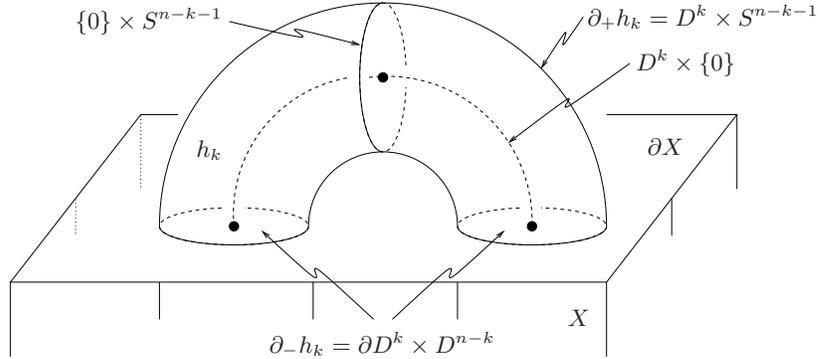}
  \caption{Attaching a $k$-handle.}
  \label{figure:k-handle}
\end{figure}

Thus, to attach a $k$-handle we take an \emph{embedding}
\[ \varphi\co \partial_- h_k=\partial D^k\times D^{n-k} \longrightarrow
\partial X,\]
i.e.\ a homeomorphism onto its image, and then form the quotient space
$(X\cup h_k)/\sim$,
where $\partial_-h_k \ni p\sim\varphi(p)\in\partial X$.
Notice that
\[ \partial \bigl((X\cup h_k)/\sim\bigr)=
\bigl(\partial X\setminus\varphi(\Int(\partial_-h_k))\bigl)
\cup_{S^{k-1}\times S^{n-k-1}} \partial_+h_k,\]
where the `corner' $S^{k-1}\times S^{n-k-1}=\partial_-h_k\cap\partial_+ h_k$
is identified with its image under~$\varphi$. In other words,
the effect of a handle attachment on the boundary
is to remove a copy of $\partial_- h_k$ from $\partial X$ and
to replace it by a copy of~$\partial_+ h_k$.

Later on we shall use the following observation: If the so-called
\emph{attaching sphere} $\varphi (S^{k-1}\times\{0\})$ bounds a disc
in $X$, that disc forms together with the core
disc of $h_k$ a $k$-sphere in $X\cup h_k$.

Here are some examples of handle attachments.

\begin{exs}
(1) In the decomposition of $S^2$ in Figure~\ref{figure:surfaces} into
a southern and northern hemisphere, we can regard the former as
a $0$-handle and the latter as a $2$-handle: start with a copy of
$D^2=\{0\}\times D^2=h_0$ and then glue another copy of
$D^2=D^2\times\{0\}=h_2$ along its boundary to the boundary of $h_0$.

\vspace{1mm}

(2) A handle decomposition of the $2$-torus $T^2$ into a single $0$-handle,
two $1$-handles and one $2$-handle is shown in
Figure~\ref{figure:torus-handles}.

\begin{figure}[h]
\labellist
\small\hair 2pt
\pinlabel $h_0$ at 294 138
\pinlabel $h_1$ at 288 207
\pinlabel $h_1$ at 409 340
\pinlabel $h_2$ at 471 103
\endlabellist
\centering
\includegraphics[scale=.4]{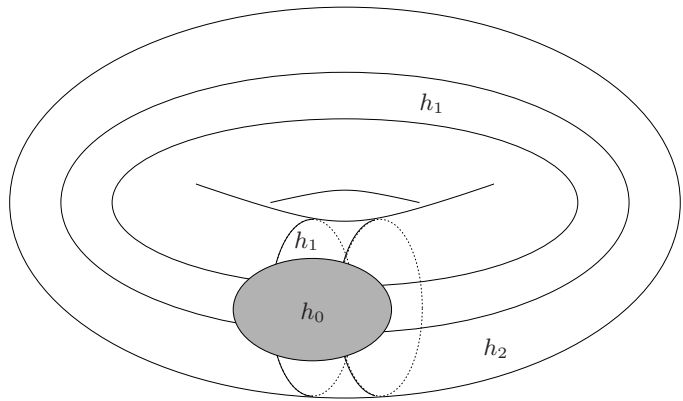}
  \caption{A handle decomposition of the $2$-torus.}
  \label{figure:torus-handles}
\end{figure}

\vspace{1mm}

(3) As shown before, the projective plane $\RP^2$ is obtained by gluing
a $2$-disc to a M\"obius band. The latter can be written as the
union of a $0$-handle and a $1$-handle, see
Figure~\ref{figure:moebius-handles}. Thus, $\RP^2$ has a handle
decomposition with a single $0$-, $1$- and $2$-handle each.
\end{exs}

\begin{figure}[h]
\labellist
\small\hair 2pt
\pinlabel $h_0$ at 72 71
\pinlabel $h_1$ [b] at 179 136
\endlabellist
\centering
\includegraphics[scale=.7]{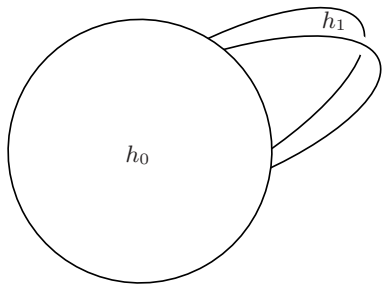}
  \caption{A handle decomposition of the M\"obius band.}
  \label{figure:moebius-handles}
\end{figure}

Notice that there are two ways to attach a $2$-dimensional $1$-handle.
In the decomposition of the $2$-torus, each $1$-handle forms an
annulus together with the $0$-handle, up to homeomorphism;
we call this an \emph{untwisted} $1$-handle. In the decomposition
of the projective plane, the $1$-handle forms a M\"obius band with
the $0$-handle; this is a \emph{twisted} $1$-handle.

You may well ask: what if we twist a $1$-handle twice or more?
I maintain that the homeomorphism type of the resulting handlebody
depends only on the number of twists being even or odd.
As an example, take an untwisted annulus $S^1\times [-1,1]$, which you can
visualise as a a circular cylinder in~$\R^3$. Slice it
open along a segment $\{*\}\times [-1,1]$ transverse to the
circular direction and then reglue it
with a double twist. The obvious bijection between the annulus
and the doubly twisted one is easily seen to be a homeomorphism.
What has changed is the embedding of the annulus in $3$-space, but not
its intrinsic topology.

Here are some relevant facts about handle decompositions:
\begin{enumerate}
\item[(H1)] Every compact \emph{smooth} manifold $M$ has a handle decomposition
with finitely many handles; this is a consequence of
Morse theory~\cite{mats02}. Here `smooth' means that the
coordinate change $\varphi_i\circ (\varphi_j)^{-1}$
between any to local charts
\[ M\supset U_i\stackrel{\varphi_i}{\longrightarrow} U_i'\subset\R^n\]
--- the coordinate change being defined on $\varphi_j(U_i\cap U_j)$ ---
should be~$C^{\infty}$. In fact, for $n\neq 4$ a handle decomposition
exists even without this smoothness assumption, cf.~\cite[I.1]{kirb89}.
In dimension four, the smoothness assumption is essential.

\item[(H2)] If the $n$-dimensional closed
manifold $M$ is connected, one can always find a
handle decomposition with precisely one handle each of index $0$ and~$n$
(and handles of intermediate index, unless the manifold is a sphere,
cf.~(H3)). Indeed,
the only handle with a disconnected lower boundary
$\partial D^k\times D^{n-k}$ is that of index $k=1$. Thus, if there
are several $0$-handles, the only way to arrive at a connected manifold
is to connect them via $1$-handles. But
two $0$-handles connected by a $1$-handle are homeomorphic to
a single $0$-handle, see Figure~\ref{figure:0-1-handle}.

For $n$-handles one argues analogously by turning the handle
decomposition `upside down'. Observe that we can read a handle
decomposition in reverse order, where every handle of index $k$ becomes a
handle of index $n-k$, and the roles of lower and upper boundary
are reversed.

\begin{figure}[h]
\labellist
\small\hair 2pt
\endlabellist
\centering
\includegraphics[scale=.4]{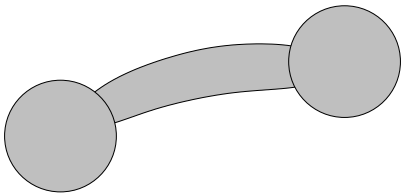}
  \caption{Two $0$-handles connected by a $1$-handle.}
  \label{figure:0-1-handle}
\end{figure}

\item[(H3)] The attaching of the $n$-handle is unique up to homeomorphism
(even up to diffeomorphism for $n\leq 6$). This is a consequence of the
Alexander trick: any homeomorphism $\phi$ of $S^{n-1}=\partial D^n$ extends
to a homeomorphism $\overline{\phi}$ of $D^n$ by setting
\[ \overline{\phi}(\bfx)=\begin{cases}
\|\bfx\|\cdot\phi(\bfx/\|\bfx\|) & \text{for $\bfx\neq\mathbf{0}$},\\
\mathbf{0}               & \text{for $\bfx=\mathbf{0}$}.
\end{cases}\]
Thus, given a handlebody $X$ with $\partial X\cong S^{n-1}$
and two gluings $X\cup_{\varphi}D^n$ and $X\cup_{\psi}D^n$,
where $\varphi,\psi$ are homeomorphisms from $S^{n-1}$ to $\partial X$,
we can define a homeomorphism
\[ X\cup_{\varphi}D^n\longrightarrow X\cup_{\psi}D^n\]
by setting it equal to the identity on $X$, and equal to the extension
of $\psi^{-1}\circ\varphi$ on~$D^n$. For $n\leq 6$ this can be
done smoothly by deep results in differential topology~\cite{kemi63},
cf.~\cite{kosi93}.
\end{enumerate}

Point (H2) shows that handle decompositions of a given manifold are far from
unique. Here is another cancellation phenomenon,
which for simplicity I explain for surfaces: start with a $2$-disc
as a $0$-handle and attach an untwisted $1$-handle to produce an annulus;
then fill in the annulus with another $2$-disc (acting as a $2$-handle)
to get back to a $2$-disc. More generally, a $k$-handle can be
`cancelled' by a $(k+1)$-handle, provided it is
possible to attach the latter appropriately.

From now on we shall always assume without loss of generality that the
manifold in question is connected, and (in the closed case) that
we have a handle decomposition with a single $0$- and $n$-handle each.
\section{Surfaces}
\label{section:surfaces}
We now apply these facts about handle decompositions
to produce $1$-dimensional representations of surfaces.
\subsection{Ignore the $2$-handle}
Point (H2) allows us to assume that we have a single $0$-handle
and a single $2$-handle;
point (H3) tells us that we need not care about the attaching
of the $2$-handle, so we may as well forgo drawing it.
Thus, the handle decompositions in
Figure~\ref{figure:surfaces-handles} may be interpreted as
pictures of $S^2$, $T^2$ and $\RP^2$, respectively.

\begin{figure}[h]
\labellist
\small\hair 2pt
\endlabellist
\centering
\includegraphics[scale=.4]{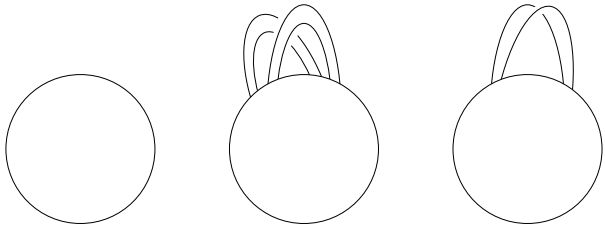}
  \caption{Handle decompositions of $S^2$, $T^2$ and $\RP^2$, not showing
the $2$-handle.}
  \label{figure:surfaces-handles}
\end{figure}

To build connected sums of these manifolds, simply repeat the
corresponding $1$-handles along the boundary of the $0$-handle.
For instance, Figure~\ref{figure:connected} shows the connected sum
$T^2\#\RP^2$. I have drawn it together with the $2$-handle so as to
indicate the circle $S^1=D^1_-\cup D^1_+$
along which the connected sum splits into
$T^2\setminus\Int(D^2)$ and $\RP^2\setminus\Int(D^2)$.

\begin{figure}[h]
\labellist
\small\hair 2pt
\pinlabel $\red{D^1_-}$ [b] at 70 141
\pinlabel $\red{D^1_+}$ [b] at 304 145
\pinlabel $h_2$ [bl] at 349 203
\pinlabel $\cup_{S^1}$ at 187 145
\endlabellist
\centering
\includegraphics[scale=.4]{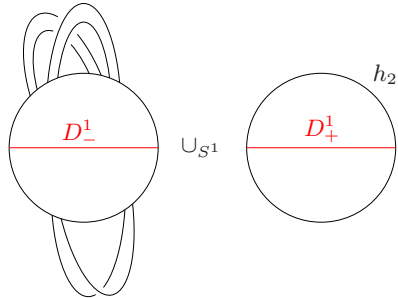}
  \caption{Handle decomposition of the connected sum $T^2\#\RP^2$.}
  \label{figure:connected}
\end{figure}
\subsection{Only draw the attaching of the $1$-handles}
As remarked upon in Section~\ref{subsection:handle}, when we have
attached a $1$-handle, in order to determine the homeomorphism
type of the resulting space we need only remember where along the
boundary of the $0$-handle we have attached the $1$-handle, and whether
it is a twisted or untwisted one. This allows us to encode a
$1$-handle by a pair of arrows in $\partial h_0=S^1$.
Think of the arrow as coming from an orientation on the
second factor in a $1$-handle $D^1\times D^1$; an untwisted
handle then corresponds to the arrows pointing in \emph{opposite}
directions along $\partial h_0=S^1$; a twisted arrow, in the
\emph{same} direction. As a further
simplification, thanks to stereographic
projection we can think of the circle $S^1$ as the union of
$\R$ with a point $\infty$ `at infinity', so we may draw the
arrows for the handle attachments on the real line; see
Figure~\ref{figure:arrows}, which illustrates this for $T^2\#\RP^2$.

\begin{figure}[h]
\labellist
\small\hair 2pt
\pinlabel $1$ [b] at 238 150
\pinlabel $2$ [b] at 275 150
\pinlabel $1'$ [b] at 326 150
\pinlabel $2'$ [b] at 366 150
\pinlabel $3$ [b] at 454 150
\pinlabel $3'$ [b] at 490 150
\endlabellist
\centering
\includegraphics[scale=.4]{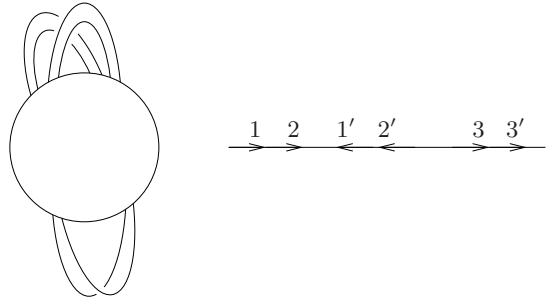}
  \caption{Handle attachments indicated by arrows.}
  \label{figure:arrows}
\end{figure}
\subsection{Handle slides}
Sliding a $1$-handle along the boundary of the handlebody to which it
is attached does not change the homeomorphism type of the
resulting space. This can be seen by undoing the slide
in a collar neighbourhood of the boundary of the handlebody,
see Figure~\ref{figure:collar}.

\begin{figure}[h]
\labellist
\small\hair 2pt
\endlabellist
\centering
\includegraphics[scale=.4]{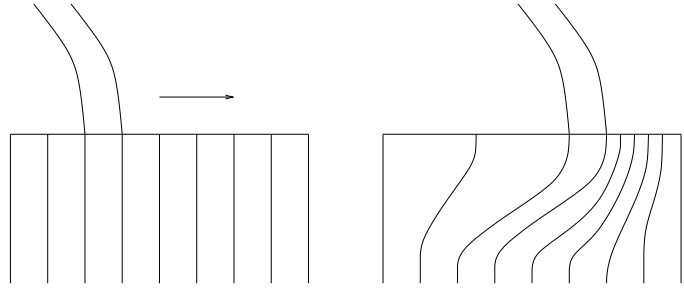}
  \caption{Sliding a $1$-handle along a boundary.}
  \label{figure:collar}
\end{figure}

In particular, we can slide an attaching region of a $1$-handle
across the point $\infty\in S^1=\partial h_0$. In the $1$-dimensional
diagram this means that an arrow can disappear on the right, say,
and reappear, pointing in the same direction, on the left.

When we slide one of the two attaching regions of a $1$-handle
$h_1^1$ along the upper boundary of a previously attached $1$-handle
$h_1^0$, the handle $h_1^1$ receives a twist
if $h_1^0$ was twisted. Notice that such a slide starts at one
attaching region of $h_1^0$ and ends at the other one, at the same
side of the partner arrow indicating the attachment of~$h_1^0$,
see Figures \ref{figure:handleslide-untwisted}
and~\ref{figure:handleslide-twisted}. In either figure, one attaching
region of the handle labelled $1$ slides across the handle whose
attaching regions are labelled $0$ and $0'$; the slide starts
at the back of the $0$-arrow and ends at the back of the $0'$-arrow.

\begin{figure}[h]
\labellist
\small\hair 2pt
\pinlabel $0$ [b] at 157 30
\pinlabel $0'$ [b] at 214 30
\pinlabel $1$ [b] at 85 30
\pinlabel $0$ [b] at 516 30
\pinlabel $0'$ [b] at 568 30
\pinlabel $1$ [b] at 643 30
\endlabellist
\centering
\includegraphics[scale=.38]{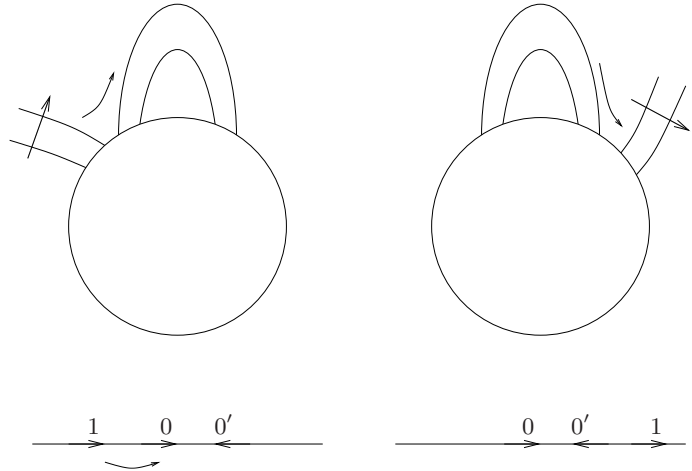}
  \caption{Sliding a $1$-handle across an untwisted $1$-handle.}
  \label{figure:handleslide-untwisted}
\end{figure}

\begin{figure}[h]
\labellist
\small\hair 2pt
\pinlabel $0$ [b] at 157 30
\pinlabel $0'$ [b] at 193 30
\pinlabel $1$ [b] at 85 30
\pinlabel $0$ [b] at 516 30
\pinlabel $1$ [b] at 572 30
\pinlabel $0'$ [b] at 608 30
\endlabellist
\centering
\includegraphics[scale=.38]{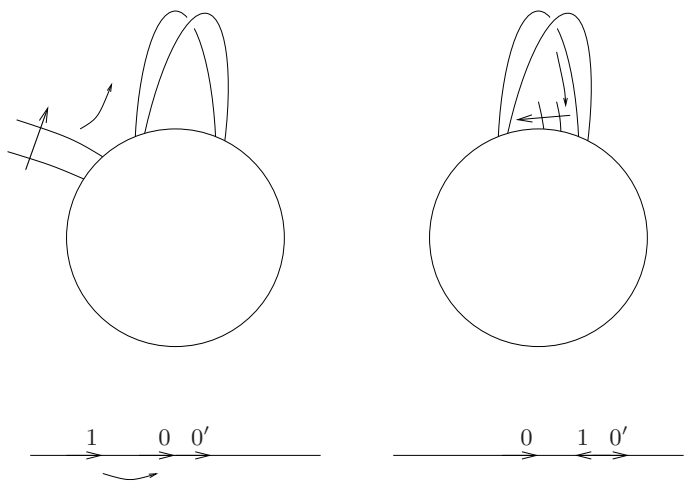}
  \caption{Sliding a $1$-handle across a twisted $1$-handle.}
  \label{figure:handleslide-twisted}
\end{figure}

Notice that by the same sliding argument we may assume (as we did
implicitly in the preceding section) that
each $1$-handle is attached to the boundary of the $0$-handle ---
disjoint from the attaching regions of the other $1$-handles ---
and not to parts of the upper boundary of other $1$-handles.

The precept that every Maths talk should contain a proof applies
equally to colloquia. Here is a fundamental statement on the
classification of surfaces that we can now prove in diagrammatic language.

\begin{prop}
\label{prop:surfaces}
The surface $T^2\#\RP^2$ is homeomorphic to $\RP^2\#\RP^2\#\RP^2$.
\end{prop}

\begin{proof}
With the handle slides we discussed, the proof reduces to the
three slides shown in the $1$-dimensional diagrams
of Figure~\ref{figure:proof-diagram}. The diagram at the top
represents $T^2\#\RP^2$; that at the bottom, $\RP^2\#\RP^2\#\RP^2$.
\end{proof}

\begin{figure}[h]
\labellist
\small\hair 2pt
\pinlabel $1$ [b] at 68 443
\pinlabel $2$ [b] at 134 443
\pinlabel $1'$ [b] at 230 443
\pinlabel $2'$ [b] at 301 443
\pinlabel $3$ [b] at 350 443
\pinlabel $3'$ [b] at 422 443
\pinlabel $1$ [b] at 68 299
\pinlabel $2$ [b] at 134 299
\pinlabel $3$ [b] at 206 299
\pinlabel $2'$ [b] at 282 299
\pinlabel $1'$ [b] at 350 299
\pinlabel $3'$ [b] at 422 299
\pinlabel $1$ [b] at 68 155
\pinlabel $3'$ [b] at 157 155
\pinlabel $3$ [b] at 230 155
\pinlabel $2$ [b] at 279 155
\pinlabel $2'$ [b] at 350 155
\pinlabel $1'$ [b] at 422 155
\pinlabel $3'$ [b] at 84 11
\pinlabel $3$ [b] at 155 11
\pinlabel $2$ [b] at 206 11
\pinlabel $2'$ [b] at 278 11
\pinlabel $1'$ [b] at 350 11
\pinlabel $1$ [b] at 422 11
\endlabellist
\centering
\includegraphics[scale=.4]{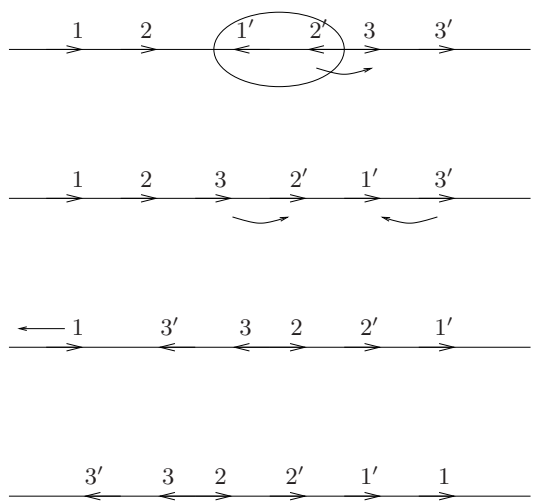}
  \caption{Proof of Proposition~\ref{prop:surfaces}.}
  \label{figure:proof-diagram}
\end{figure}

With very similar arguments one can then show that in fact every
compact surface is homeomorphic to exactly one of $\#_g T^2$, $g\in\N_0$,
or $\#_h\RP^2$, $h\in\N$. I have used this approach in
an undergraduate lecture course on surfaces~\cite{geig13}, and in my view
it is superior to the traditional approach to the classification of surfaces
via triangulations and cut-and-paste arguments, since it extends more
naturally to higher dimensions, and the steps one has to take
to simplify a given word for the attaching of $1$-handles into
one of the standard words for the model surfaces are quite straightforward.

\section{Dimension three: Heegaard diagrams}
As we have seen, there are two ways to attach a $2$-dimensional
$1$-handle to the boundary of a $2$-disc. The same is true in
higher dimensions. Indeed, the lower boundary $\partial_- h_1=
\partial D^1\times D^{n-1}$ of an $n$-dimensional $1$-handle consists of two
$(n-1)$-discs with opposite orientations. An untwisted $1$-handle
corresponds to an orientation-preserving embedding of $\partial_-h_1$
into the boundary of the $0$-handle.

From now on I shall assume that all manifolds are orientable. This
is equivalent to saying that all $1$-handles are untwisted. The
$3$-dimensional situation is shown in Figure~\ref{figure:3-dim-1-handle}.
With this assumption understood, each $1$-handle $h_1$ is then
simply represented by a pair of
\emph{attaching discs} in the boundary $\partial h_0$ of the $0$-handle,
i.e.\ the image of the lower boundary $\partial_-h_1$ under the
gluing map.

\begin{figure}[h]
\labellist
\small\hair 2pt
\endlabellist
\centering
\includegraphics[scale=.5]{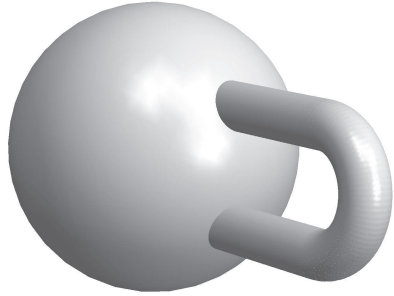}
  \caption{A $3$-dimensional $0$-handle with an untwisted $1$-handle
attached.}
  \label{figure:3-dim-1-handle}
\end{figure}

Notice that when we attach $g$ untwisted $3$-dimensional $1$-handles to the
boundary of a $3$-ball, the boundary of the resulting handlebody
is the oriented surface $\Sigma_g$ of genus~$g$.
Figure~\ref{figure:3-dim-handlebody} shows an alternative view of such
a handlebody (for $g=2$).

\begin{figure}[h]
\labellist
\small\hair 2pt
\pinlabel $h_0$ at 353 168
\pinlabel $h_1$ [l] at 147 102
\pinlabel $h_1$ [l] at 489 102
\endlabellist
\centering
\includegraphics[scale=.4]{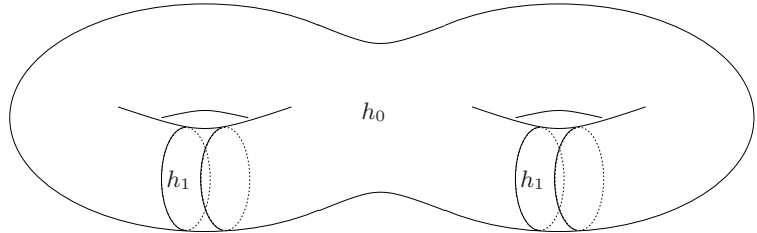}
  \caption{A $3$-dimensional $0$-handle with two untwisted $1$-handles
attached --- an alternative view.}
  \label{figure:3-dim-handlebody}
\end{figure}

Now let $M$ be a closed $3$-manifold. Again we appeal to (H2) and
consider a handle decomposition of $M$ having precisely one
$0$-handle and one $3$-handle. Write $M_1$ for the handlebody made up
of the $0$-handle and all (say~$g$) $1$-handles. This is a manifold
with boundary~$\Sigma_g$. The complement $M\setminus \Int (M_1)$
consists of the $2$-handles and the $3$-handle. As observed in (H2),
the $2$-handles are glued to the $3$-handle along their upper
boundary, so we may regard this complement as a $0$-handle with
$1$-handles attached.

Since the boundary of $M_1$ equals that of the complement $M\setminus
\Int (M_1)$, the number of $2$-handles must likewise be~$g$. In other words,
$M$ is obtained by gluing two copies of a $1$-handlebody of genus $g$
by a homeomorphism of their boundary~$\Sigma_g$, see
Figure~\ref{figure:heegaard} (for $g=2$). Such a decomposition
of a $3$-manifold is called a \emph{Heegaard splitting}.

\begin{figure}[h]
\labellist
\small\hair 2pt
\pinlabel $M_1$ [br] at 28 111
\pinlabel $M\setminus\Int(M_1)$ [br] at 461 111
\pinlabel $h_0$ at 164 75
\pinlabel $h_3$ at 598 75
\pinlabel $h_1$ [t] at 83 24
\pinlabel $h_1$ [t] at 238 24
\pinlabel $h_2$ [t] at 513 24
\pinlabel $h_2$ [t] at 670 24
\pinlabel $\cup_{\Sigma_2}$ at 382 72
\endlabellist
\centering
\includegraphics[scale=.45]{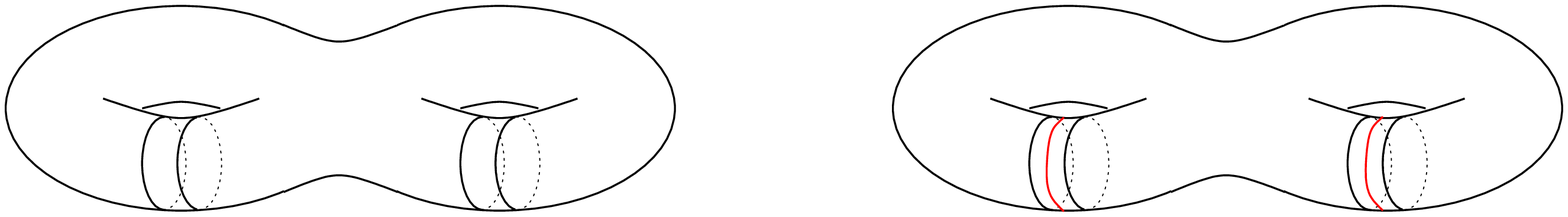}
  \caption{A Heegaard splitting of genus~$2$.}
  \label{figure:heegaard}
\end{figure}

A $2$-handle $h_2=D^2\times D^1$ is attached to $M_1$ by an
embedding $\varphi$ of the lower boundary $\partial_-h_2=
\partial D^2\times D^1$ into~$\partial M_1$. Up to inessential choices,
it suffices to draw the image $\varphi(\partial D^2\times\{0\})
\subset\partial M_1$ of the central circle in~$\partial_-h_2$, i.e.\
the \emph{attaching circle} of the $2$-handle.
Parts of this attaching circle will lie in the upper boundary
$D^1\times\partial D^2$ of the $1$-handles, but we may assume that 
the curve traverses these upper boundaries along a straight line segment
$D^1\times\{*\}$ (or possibly several disjoint line segments
of this form). It then suffices to draw the part of the
attaching circle that lies in the boundary of the $0$-handle~$h_0$, outside
the attaching discs of the $1$-handles.

The final simplification comes again from regarding $\partial h_0=S^2$
as the union of $\R^2$ with a point $\infty$ `at infinity';
this allows us to draw all the information about the attaching
of $1$- and $2$-handles in~$\R^2$. Such a planar diagram for
a $3$-manifold is called a \emph{Heegaard diagram}.
Here are some simple examples.

\begin{exs}
(1) The $3$-sphere $S^3$ can be obtained
by gluing two $3$-balls along their boundary, that is, a
$3$-dimensional $0$-handle and a $3$-handle. Since there are no
$1$- or $2$-handles, this handle decomposition corresponds to
the empty diagram in~$\R^2$.

\vspace{1mm}

(1') The $3$-sphere can also be given a Heegaard splitting of
any other genus. To describe the genus~$1$ splitting,
think of $S^3$ as the unit sphere in~$\C^2$. The subset
\[ V_1=\{ (z_1,z_2)\in S^3\co |z_2|\leq 1/\sqrt{2}\}\]
is homeomorphic to a solid torus $S^1\times D^2$ via
$\phi_1\co (z_1,z_2)\mapsto (z_1/|z_1|,\sqrt{2}\,z_2)$; similarly, the
complement $V_2=S^3\setminus\Int(V_1)$ is homeomorphic to
$D^2\times S^1$ via $\phi_2\co (z_1,z_2)\mapsto (\sqrt{2}\, z_1,z_2/|z_2|)$.

We regard $V_1$ as the union of a $0$- and a $1$-handle; $V_2$
is the union of a $2$-handle and a $3$-handle. The gluing
of $V_1\cong S^1\times D^2$ and $V_2\cong D^2\times S^1$ that produces
the $3$-sphere is given by the obvious identification
\[ \partial V_1=S^1\times\partial D^2=S^1\times S^1=\partial D^2\times S^1
=\partial V_2,\]
since $\phi_1=\phi_2$ on $\partial V_1\cap\partial V_2$.

On the boundary $\partial V$ of a solid torus $V=S^1\times D^2$, a curve of
the form $\{*\}\times\partial D^2$ is called a \emph{meridian}.
It is characterised (up to smooth deformations) by the
fact that it is not contractible inside~$\partial V$, but it
bounds a disc in~$V$. A curve on $\partial V$ that, when viewed
as a curve in $V$, goes once along the $S^1$-factor, e.g.\
the curve $S^1\times\{*\}$, is called a \emph{longitude}
of the solid torus.
In this language, the gluing of $V_1$ and $V_2$ that
produces $S^3$ identifies
a meridian $\mu_2$ of $V_2$ with a longitude $\lambda_1$ of~$V_1$.

These curves are shown in Figure~\ref{figure:s3-heegaard1}.
We may think of $\mu_2$ as $\partial D^2\times\{0\}\subset \partial_-h_2$;
its image $\lambda_1$ on $\partial V_1=\partial(h_0\cup h_1)$ is then the
attaching circle of the $2$-handle.

\begin{figure}[h]
\labellist
\small\hair 2pt
\pinlabel $V_1$ [r] at 0 455
\pinlabel $V_2$ [r] at 0 154
\pinlabel $h_0$ at 124 532
\pinlabel $h_1$ [tl] at 167 423
\pinlabel ${h_2\cup h_3}$ at 200 210
\pinlabel $\mu_2$ [l] at 160 92
\pinlabel $\lambda_1$ [l] at 345 447
\endlabellist
\centering
\includegraphics[scale=.4]{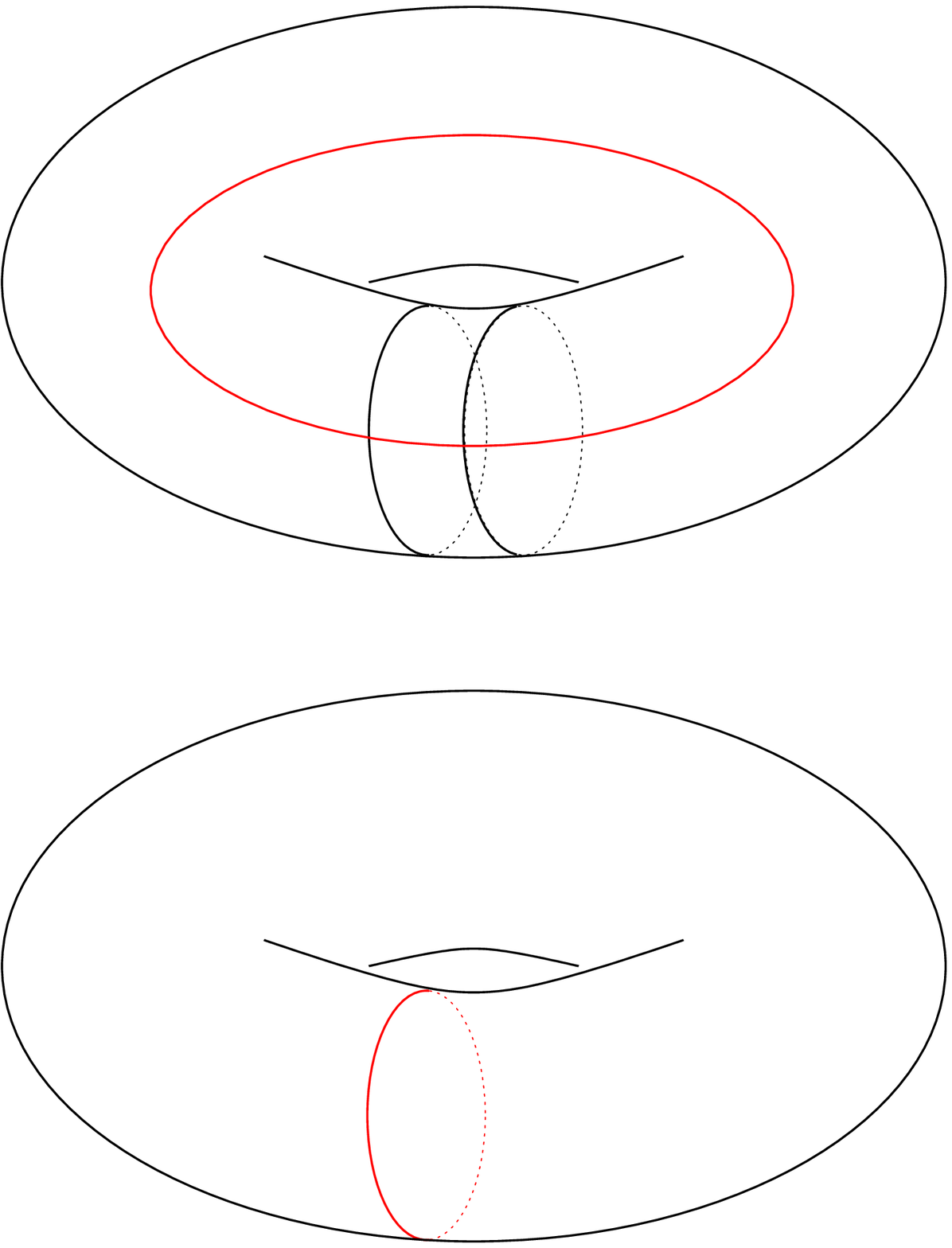}
  \caption{A Heegaard splitting of $S^3$ of genus~$1$.}
  \label{figure:s3-heegaard1}
\end{figure}

Alternatively, consider Figure~\ref{figure:s3-heegaard1a}.
Here we think of $S^3$ as $\R^3\cup\{\infty\}$, and $\R^3$
is regarded as the result of rotating the drawing plane
about a vertical axis in the plane. The two discs represent
the solid torus~$V_2$.
Each line connecting the discs represents a disc; these discs
make up the solid torus~$V_1$. Here the gluing of $\mu_2$ with
$\lambda_1$ is perfectly transparent.

\begin{figure}[h]
\labellist
\small\hair 2pt
\pinlabel $V_1$ at 593 372
\pinlabel $V_2$ at 560 217
\pinlabel $\mu_2=\lambda_1$ [bl] at 627 243
\pinlabel $\text{disc}$ [bl] at 10 220
\pinlabel ${\text{through $\infty$}}$ [tl] at 10 217
\pinlabel ${\text{circle through $\infty$}}$ [l] at 379 21
\endlabellist
\centering
\includegraphics[scale=.4]{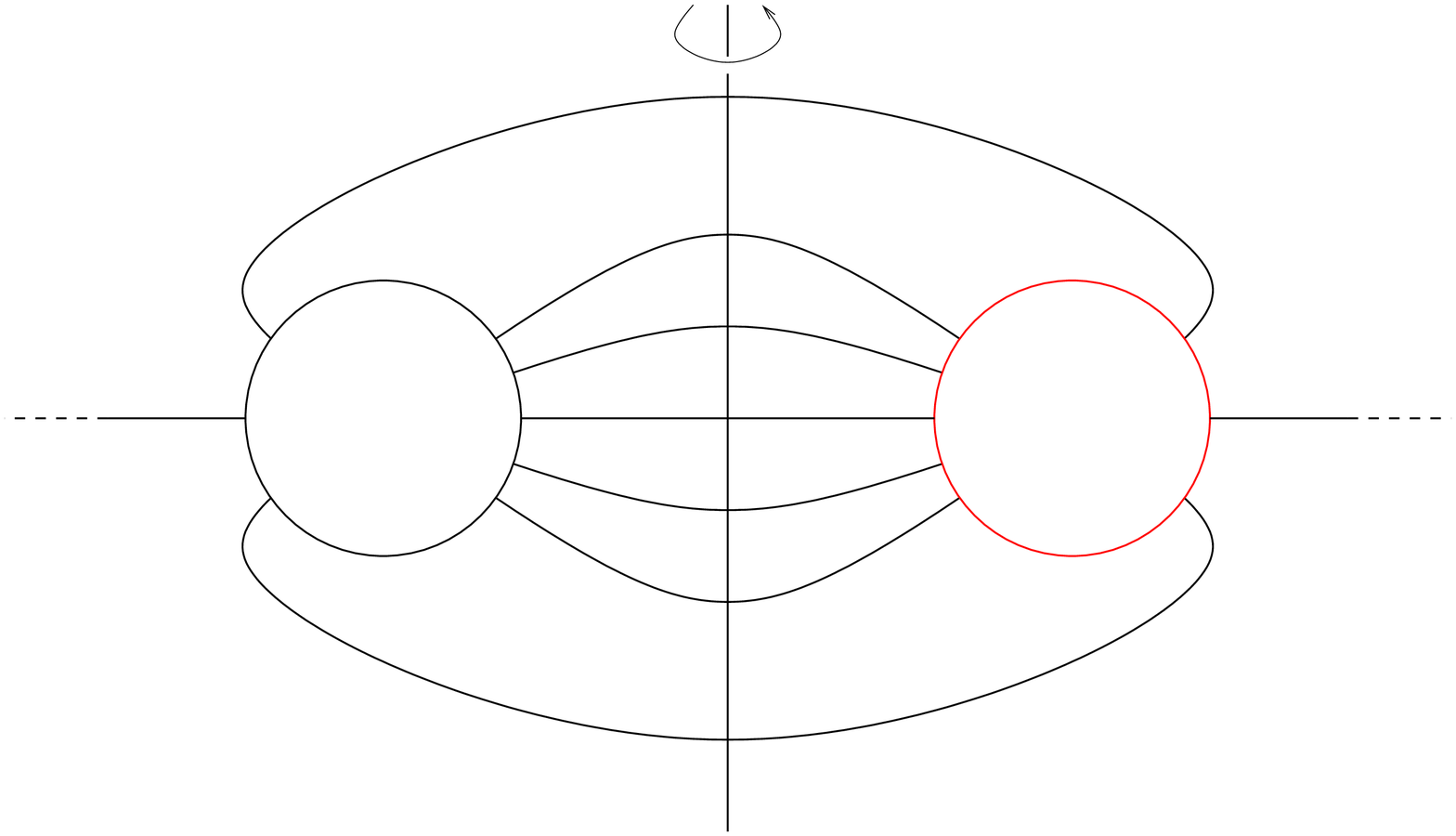}
  \caption{A Heegaard splitting of $S^3$ of genus~$1$ ---
an alternative view.}
  \label{figure:s3-heegaard1a}
\end{figure}

The translation of this information into a Heegaard diagram for $S^3$
is shown in Figure~\ref{figure:s3-heegaard1b}.
You see the attaching discs of the $1$-handle
in $\partial h_0=S^2=\R^2\cup\{\infty\}$; the line connecting them
is the part of the attaching circle that lies outside $\partial_+h_1$.

\begin{figure}[h]
\labellist
\small\hair 2pt
\endlabellist
\centering
\includegraphics[scale=.4]{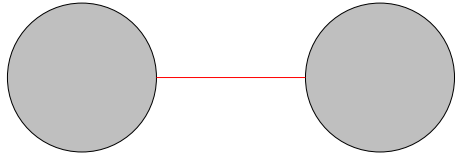}
  \caption{The Heegaard diagram of $S^3$ corresponding to the
splitting in Figures~\ref{figure:s3-heegaard1} and~\ref{figure:s3-heegaard1a}.}
  \label{figure:s3-heegaard1b}
\end{figure}

\vspace{1mm}

(2) Another simple example of a $3$-manifold is the product $S^1\times S^2$.
This, too, has a Heegaard splitting of genus~$1$. Take two copies
of a solid torus $S^1\times D^2$ and identify them along the
boundary with the identity map of $S^1\times\partial D^2$,
see Figure~\ref{figure:s1s2-split}.

\begin{figure}[h]
\labellist
\small\hair 2pt
\pinlabel $h_0$ at 54 418
\pinlabel $h_1$ [tl] at 165 371
\pinlabel ${h_2\cup h_3}$ at 200 55
\pinlabel $\mu_2$ [l] at 217 181
\pinlabel ${\text{image of $\mu_2$}}$ [l] at 217 462
\endlabellist
\centering
\includegraphics[scale=.4]{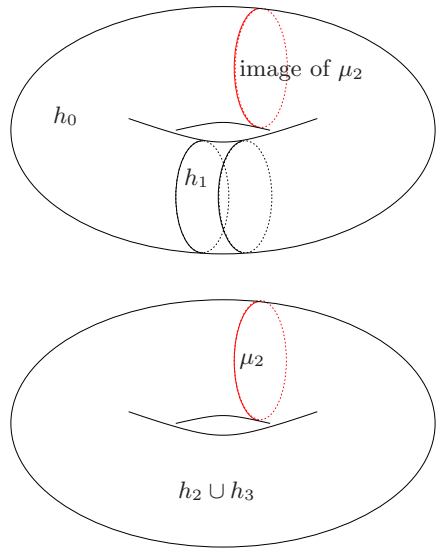}
  \caption{A Heegaard splitting of $S^1\times S^2$.}
  \label{figure:s1s2-split}
\end{figure}

The identification of the two solid tori with $h_0\cup h_1$ and
$h_2\cup h_3$, respectively, can be done in such a way
that the attaching circle of the $2$-handle, i.e.\
the image of $\mu_2=\partial D^2\times\{0\}\subset\partial_-h_2$
in $\partial (h_0\cup h_1)$, lies outside $\partial_+h_1$.
Then the corresponding Heegaard diagram is given by
Figure~\ref{figure:s1s2-diagram}.
\end{exs}

\begin{figure}[h]
\labellist
\small\hair 2pt
\endlabellist
\centering
\includegraphics[scale=.4]{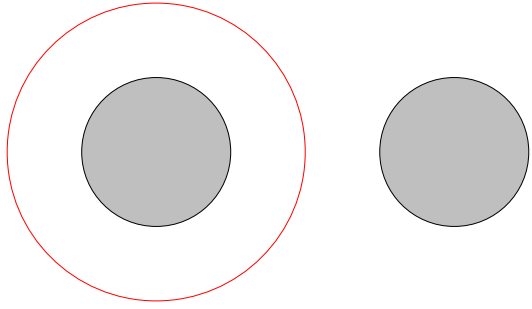}
  \caption{The Heegaard diagram of $S^1\times S^2$ corresponding to the
splitting in Figure~\ref{figure:s1s2-split}.}
  \label{figure:s1s2-diagram}
\end{figure}

For more information on Heegaard splittings and diagrams
see~\cite{prso97}.
\section{Dimension four: Kirby diagrams}
Going up one dimension, we are hopeful to represent
$4$-dimensional manifolds by $3$-dimensional diagrams. However, there
seems to be one serious impediment.
The lower boundary of a $4$-dimensional $3$-handle $D^3\times D^1$
is $\partial D^3\times D^1$, in other words, a thickened $2$-sphere.
Thus, in order to describe the attaching of $3$-handles, it seems that we
face the well-nigh impossible task to understand --- and draw --- embeddings
of $2$-spheres into the boundary of a $4$-dimensional handlebody made up of
handles of index at most~$3$.

One of the great serendipities of $4$-manifold topology, as we shall see,
is that the $3$-handles do not, in fact, carry any topological information.
This allows us to represent any closed $4$-dimensional manifold
solely by the handles of index at most~$2$.
\subsection{Kirby diagrams of $2$-handlebodies}
Before I explain this phenomenon, I shall present a few examples
of $4$-dimensio\-nal $2$-handlebodies, that is, manifolds with
boundary made up of handles of index at most~$2$.

\begin{exs}
(1) A $4$-dimensional $0$-handle $h_0$ is simply a copy of the $4$-disc~$D^4$;
its boundary $\partial h_0$ is a $3$-sphere~$S^3$. In order to depict
the attaching of handles of higher index, we think of $S^3$ as
$\R^3\cup\{\infty\}$, and draw the attaching regions of the
handles in~$\R^3$. These $3$-dimensional pictures are called
\emph{Kirby diagrams}. Thus, the empty diagram represents
the $4$-disc.

\vspace{1mm}

(2) A $1$-handle $h_1$ is a copy of $D^1\times D^3$, attached to $\partial h_0$
along $\partial_-h_1=\partial D^1\times D^3$, i.e.\ two solid balls, see
Figure~\ref{figure:s1d3}. Again we assume that the two balls represent
an orientation-preserving embedding of $\partial_-h_1$, so that
the $1$-handle is untwisted.

\begin{figure}[h]
\labellist
\small\hair 2pt
\endlabellist
\centering
\includegraphics[scale=.4]{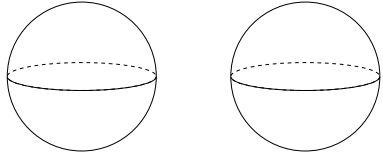}
  \caption{The Kirby diagram showing a single $1$-handle.}
  \label{figure:s1d3}
\end{figure}

The schematic picture in Figure~\ref{figure:s1d3-schematic}
shows that the manifold obtained by attaching an untwisted $1$-handle
to $D^4$ is homeomorphic to~$S^1\times D^3$.

\begin{figure}[h]
\labellist
\small\hair 2pt
\pinlabel $D^1$ [t] at 277 314
\pinlabel $D^3$ [l] at 386 242
\pinlabel $D^4$ at 277 242
\pinlabel $D^1\times{D^3}$ at 277 50
\pinlabel $D^1$ [tl] at 469 65
\endlabellist
\centering
\includegraphics[scale=.4]{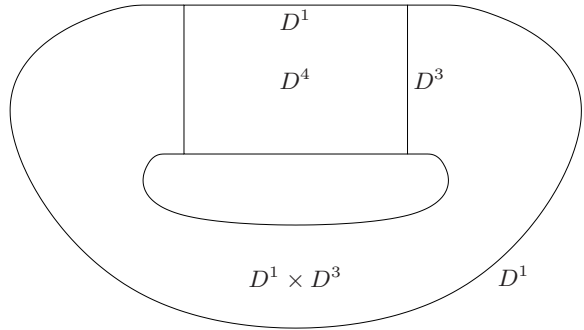}
  \caption{The Kirby diagram in Figure~\ref{figure:s1d3} represents
           $S^1\times D^3$.}
  \label{figure:s1d3-schematic}
\end{figure}

Alternatively and a little more formally, we can see this handle
decomposition by splitting $S^1$ into two intervals $D^1_{\pm}$,
glued along two points, i.e.\ a $0$-dimensional sphere~$S^0$:
\[ S^1\times D^3=(D^1_-\cup_{S^0}D^1_+)\times D^3=
(D^1_-\times D^3)\cup_{S^0\times D^3} (D^1_+\times D^3)=
h_0\cup h_3.\]

\vspace{1mm}

(3) In order to describe the attaching of a $2$-handle
$h_2=D^2\times D^2$ to a $4$-dimensional
handlebody~$X$, we need to visualise the image
of $\partial_-h_2=\partial D^2\times D^2$ in $\partial X$ under the
embedding used for the gluing. In the case of a
$3$-dimensional handle, it was sufficient to describe
the image of $\partial D^2\times\{0\}\subset\partial D^2\times D^1$,
since the embedding of $\partial D^2\times\{0\}$ extends to one
of $\partial D^2\times D^1$ in an essentially unique way.

In the $4$-dimensional situation, however, there are
homeomorphisms of $\partial D^2\times D^2$ that rotate the
$D^2$-factor as we go along the circle~$\partial D^2$, and that do not
extend to homeomorphisms of $D^2\times D^2$.
Regarding $D^2$ as a subset of~$\C$, for
each integer $m$ we have the homeomorphism
\[ \begin{array}{ccc}
\partial D^2\times D^2 & \longrightarrow & \partial D^2\times D^2\\
(\rme^{\rmi\theta}, z) & \longmapsto     & (\rme^{\rmi\theta},
                                            \rme^{\rmi m\theta}z),
\end{array}\]
and we can change a given embedding $\partial D^2\times D^2
\rightarrow\partial X$ by precomposing with this homeomorphism.
It is not hard to see, however, that this is
the only available freedom, up to inessential choices.
This implies that the attaching map is determined by the
image of $\partial D^2\times\{0\}$ and the parallel circle
$\partial D^2\times\{1\}$.

Figure~\ref{figure:plus1twist} shows a right-handed twist
of two curves in~$\R^3$. A box with an integer $m$ in a diagram
of knots (i.e.\ embedded circles) in $\R^3$ stands for
$m$ right-handed twists for $m>0$, and $|m|$ left-handed twists
for $m<0$.

\begin{figure}[h]
\labellist
\small\hair 2pt
\endlabellist
\centering
\includegraphics[scale=.4]{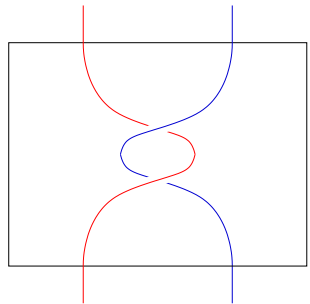}
  \caption{A right-handed twist.}
  \label{figure:plus1twist}
\end{figure}

For instance, the Kirby diagram consisting of a single unknotted circle
labelled with the integer $m$ is meant to represent the attaching
of a $2$-handle to $\partial D^4$ with the images of $\partial D^2\times\{0\}$
and $\partial D^2\times\{1\}$ in $\partial D^4$
as shown in Figure~\ref{figure:d2bundle1}.

\begin{figure}[h]
\labellist
\pinlabel $m$ [bl] at 125 140
\pinlabel $=$ at 208 91
\pinlabel $m$ at 441 91
\pinlabel ${\partial D^2\times\{1\}}$ [bl] at 424 156
\pinlabel ${\partial D^2\times\{0\}}$ [bl] at 300 41
\small\hair 2pt
\endlabellist
\centering
\includegraphics[scale=.4]{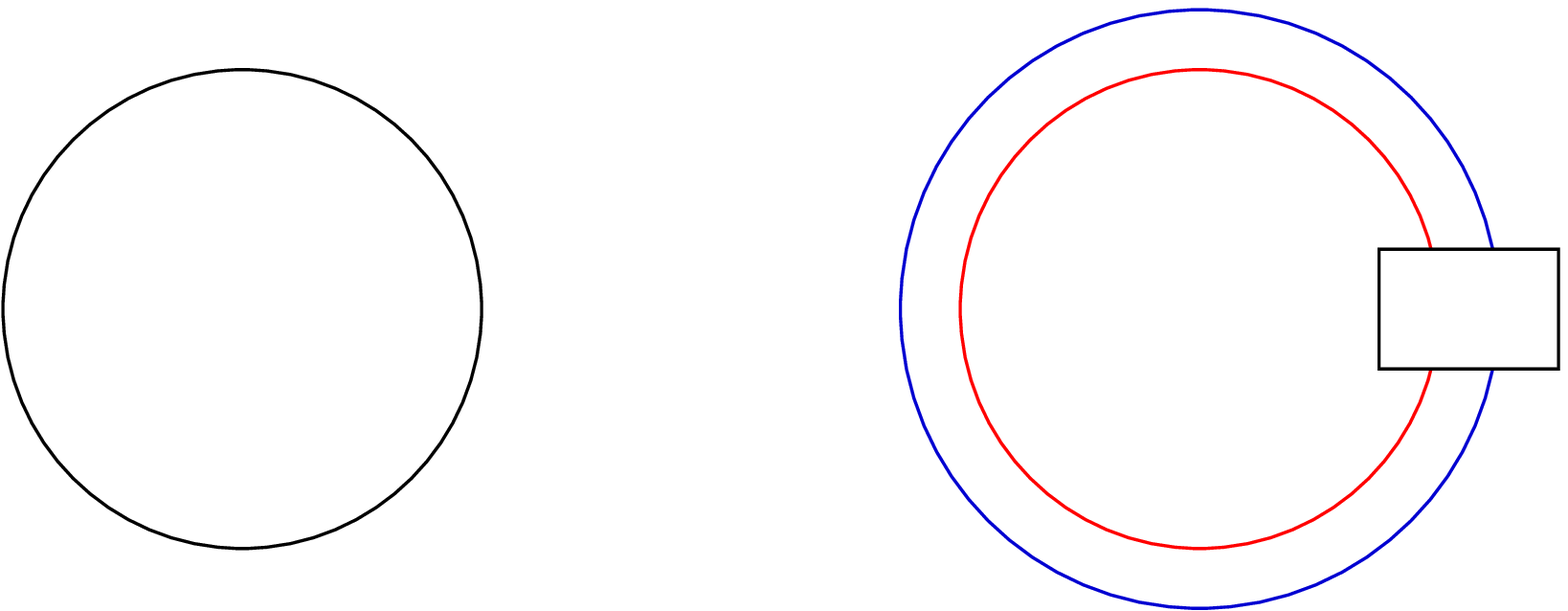}
  \caption{A single $2$-handle attached along an unknot.}
  \label{figure:d2bundle1}
\end{figure}

Now consider a schematic picture as in Figure~\ref{figure:s1d3-schematic},
but with $h_0=D^4$ thought of as $D^2\times D^2$, and a $2$-handle attached
along the unknot $\partial D^2\times\{0\}\subset D^2\times D^2=h_0$.
The horizontal disc $D^2\times\{0\}\subset D^4$ and the core
disc $D^2\times\{0\}$ in the $2$-handle $h_2=D^2\times D^2$ form
a $2$-sphere $S^2_0$. Transverse to $S^2_0$ we see copies of~$D^2$,
so we have a well-defined projection $h_0\cup h_2\rightarrow S^2$,
with the preimage of each point in $S^2$ equal to a $2$-disc.
This is what is called a $D^2$-bundle over~$S^2$.

The disc $D^2\times\{1\}$ in the $2$-handle can likewise be completed
to a $2$-sphere: simply observe that the image of its boundary
$\partial D^2\times\{1\}$ is another unknot in $\partial h_0$
that bounds a disc. Call this $2$-sphere $S^2_1$.

I claim that $S^2_0$ and $S^2_1$, after `wiggling',
intersect transversely in exactly $|m|$ points. To see this, we draw yet
another schematic picture. Figure~\ref{figure:s-intersection} shows
what we should do. The image $K_i$ of $\partial D^2\times\{i\}$,
$i=0,1$, bounds a disc $D^2_i$ in $\partial D^4$, which
together with the disc $D^2\times\{i\}$
in the $2$-handle forms the sphere $S^2_i$.
However, instead of taking $D^2_i$ in $\partial D^4$
to form the sphere, we can push it vertically into $D^4$ and
connect it with the disc in the $2$-handle via a vertical
cylinder over~$K_i$. When we push $D^2_0$ deeper into
$D^4$ than~$D^2_1$, the intersection points of $S^2_0$ and $S^2_1$
will be in one-to-one correspondence with the intersection points
of $K_0$ and the original $D^2_1\subset\partial D^4$.
With the obvious choice of $D^2_1$, an essentially flat disc
with boundary $K_1=\partial D^2\times\{1\}$ in Figure~\ref{figure:d2bundle1},
there are $|m|$ such intersection points. One can in fact endow
the $2$-spheres with orientations and count the intersection points
with sign; the correct count then will be~$m$, independent of
the choice of `wiggling'.
This number determines the $D^2$-bundle over $S^2$ and is called
its \emph{Euler number}.

\begin{figure}[h]
\labellist
\small\hair 2pt
\pinlabel $h_0=D^4$ at 142 146
\pinlabel $h_2$ at 434 73
\pinlabel $S^2_0$ [bl] at 385 220
\pinlabel $S^2_1$ [bl] at 466 232
\endlabellist
\centering
\includegraphics[scale=.4]{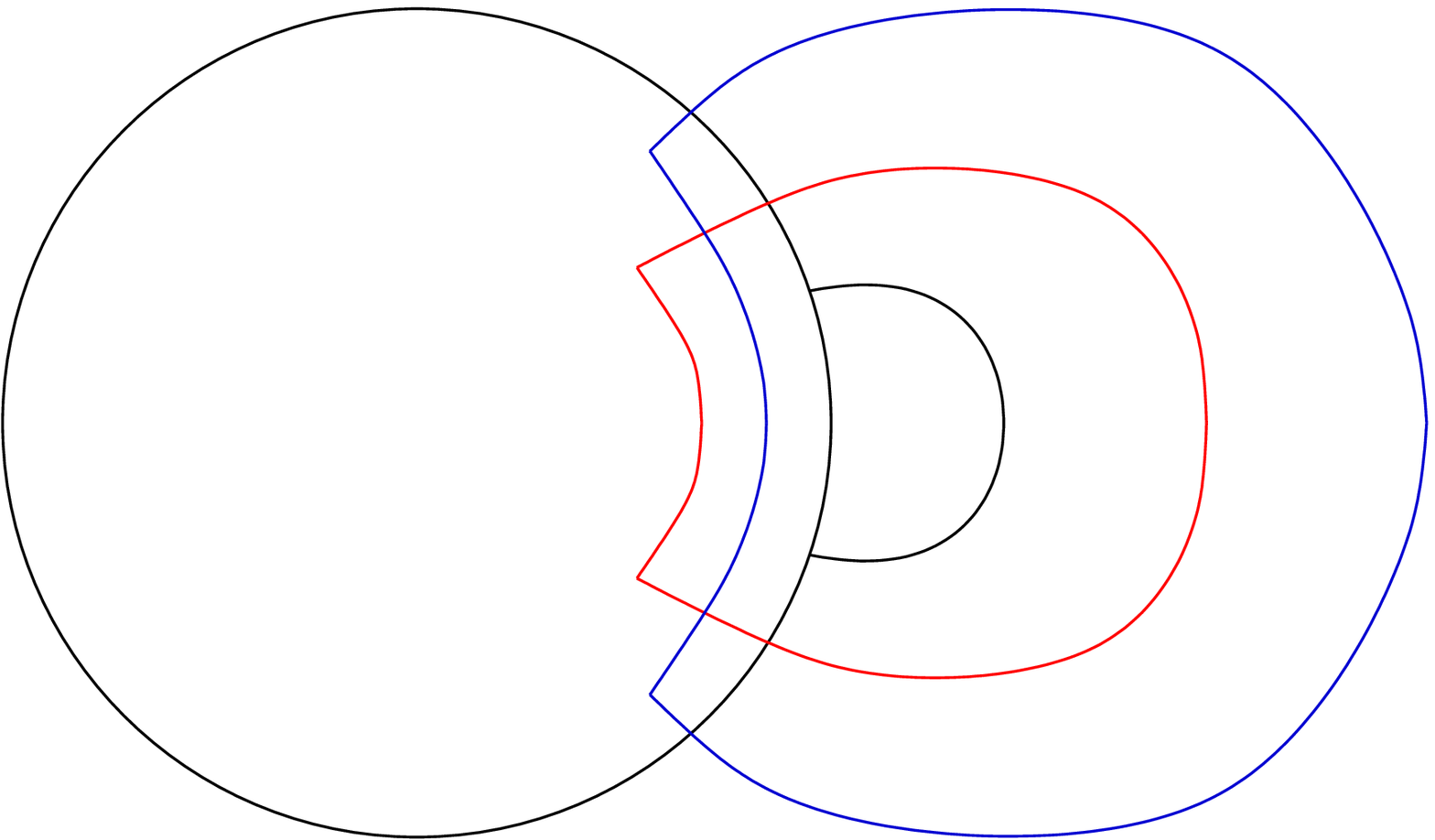}
  \caption{The intersection of $S^2_0$ and $S^2_1$.}
  \label{figure:s-intersection}
\end{figure}

\end{exs}
\subsection{Kirby diagrams of closed $4$-manifolds}
When we turn our attention to \emph{closed} $4$-manifolds, we also
need to comprehend the effect of attaching $3$-handles.
Regard the $3$-handles and the single $4$-handle as a $0$-handle
with $1$-handles attached. What does this handlebody look like?

Let us first consider the $3$-dimensional situation.
When we attach a single untwisted $3$-dimensional $1$-handle to a $3$-ball,
we obtain a solid torus $S^1\times D^2$. Attaching a further
$1$-handle amounts to starting with two solid tori
and identifying a pair of $2$-discs, one each in the
boundaries of the solid tori.
This construction is called a \emph{boundary connected sum},
denoted by~$\natural$. The effect on the boundary is that of a connected
sum in the old sense.
Repeating this process with
further $1$-handles ($g$ in total), we obtain the manifold
$\natural_g S^1\times D^2$ with boundary $\#_g S^1\times S^1=\Sigma_g$.

Returning to the $4$-dimensional situation, when we attach $g$ (untwisted)
$1$-handles to a $4$-ball, we obtain $\natural_g S^1\times D^3$
with boundary $\#_g S^1\times S^2$. The felicitous fact alluded to before is
the following theorem due to Laudenbach and Po\'enaru~\cite{lapo72}.

\begin{thm}
Any diffeomorphism of $\#_g S^1\times S^2$ extends to one
of\/ $\natural_g S^1\times D^3$.
\end{thm}

By the argument as in Section~\ref{section:intro}~(H3), this implies that a
closed $4$-manifold is determined, up to diffeomorphism~(!), by the $1$- and
$2$-handles.
Beware that the boundary of a $4$-dimensional $2$-handlebody will not,
in general, be diffeomorphic to $\#_g S^1\times S^2$,
so not every $2$-handlebody corresponds to a closed $4$-manifold.
If, however, the boundary is indeed of that form, the $2$-handlebody
represents a unique closed smooth manifold.

In stark contrast with dimensions two and three, there is
a subtle difference between the classification of topological
and smooth $4$-manifolds, see \cite{gost99} and~\cite{kirb89}.
Some topological $4$-manifolds do not admit
any smooth structures; others, infinitely many. Strikingly,
while Euclidean spaces in dimensions other than four admit
a unique smooth structure,
$\R^4$ admits uncountably many distinct smooth structures, up
to diffeomorphism.

It therefore deserves to be emphasised that a Kirby diagram
of a closed $4$-manifold represents a unique smooth
$4$-manifold up to diffeomorphism.

\begin{exs}
(1) The $4$-sphere is given by gluing two $4$-balls along their boundary,
that is, $S^4=h_0\cup h_4$. So the empty Kirby diagram, read as a diagram of
a closed $4$-manifold, represents~$S^4$.

\vspace{1mm}

(2) The Kirby diagram containing a single $1$-handle describes
$h_0\cup h_1=S^1\times D^3$ with boundary $S^1\times S^2$.
Thus, in the corresponding closed $4$-manifold we must
have a single $3$-handle, and the gluing of
the two copies $h_0\cup h_1$
and $h_3\cup h_4$ of $S^1\times D^3$ yields $S^1\times S^3$.

\vspace{1mm}

(3) The diagram in Figure~\ref{figure:d2bundle1} with $m=1$
represents the complex projective plane $\CP^2$ with its
natural orientation coming from the complex structure.
This can be seen as follows. The complement of $\CP^1\subset\CP^2$
is an open $4$-ball: simply take $\CP^1$ to be the complex line at infinity. 
A neighbourhood of this complex line $\CP^1\cong S^2$ is
a $D^2$-bundle over~$S^2$. Consequently, $\CP^2$ is the result
of attaching a $4$-handle to this bundle space. The Euler number of
the disc bundle in question being $m=1$ corresponds with the fact
that $\CP^1\subset\CP^2$ has self-intersection equal to~$1$.

Similarly, the diagram in Figure~\ref{figure:d2bundle1}
with $m=-1$ represents~$\overline{\CP}^2$, the complex projective
plane with the opposite of its natural orientation.

\vspace{1mm}

(4) Decompose the $2$-sphere as $S^2=D^2_-\cup_{S^1}D^2_+$. Then the
product $4$-manifold $S^2\times S^2$ splits as
\begin{eqnarray*}
S^2\times S^2 & = & \bigl(D^2_-\cup_{S^1}D^2_+\bigr)\times
                    \bigl(D^2_-\cup_{S^1}D^2_+\bigr)\\
              & = & \bigl(D^2_-\times D^2_-\bigr)\cup_{S^1\times D^2_-}
                    \bigl(D^2_+\times D^2_-\bigr)\cup_{D^2_-\times S^1}
                    \bigl(D^2_-\times D^2_+\bigr)\\
              &   & \hphantom{xxx}
                    \cup_{D^2_+\times S^1\cup_{S^1\times S^1}S^1\times D^2_+}
                    \bigl(D^2_+\times D^2_+\bigr)\\
              & = & h_0\cup h_2\cup h_2\cup h_4,
\end{eqnarray*}
analogous to the handle decomposition of the $2$-torus $S^1\times S^1$.
The attaching circles of the $2$-handles are $S^1\times\{0\}$ and
$\{0\}\times S^1$ in $\partial(D^2_-\times D^2_-)=\partial D^4=S^3$. These
are unknotted circles, since the obvious discs they bound in $D^4$ can
be pushed to the boundary~$S^3$. You may convince yourself
that the attaching circles are linked as shown in
Figure~\ref{figure:hopf-link}; this is called the \emph{Hopf link}.

\begin{figure}[h]
\labellist
\small\hair 2pt
\pinlabel $0$ [br] at 43 247
\pinlabel $0$ [bl] at 473 177
\endlabellist
\centering
\includegraphics[scale=.25]{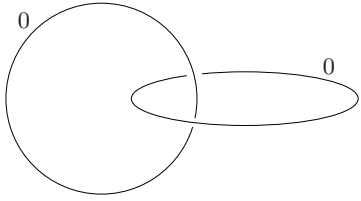}
  \caption{The Kirby diagram for $S^2\times S^2$.}
  \label{figure:hopf-link}
\end{figure}

The two $2$-spheres corresponding to the $2$-handles,
\[ \bigl(D^2_-\times\{0\}\bigr)\cup_{S^1\times\{0\}}
\bigl(D^2_+\times\{0\}\bigr)
=S^2\times\{0\}\]
and $\{0\}\times S^2$ intersect in the single point $(0,0)$. This
is consistent with the argument in Figure~\ref{figure:s-intersection},
since either unknot in the Hopf link intersects the obvious disc
bounded by the other one transversely in a single point.
The attaching of either $2$-handle produces a trivial $D^2$-bundle
$S^2\times D^2_-$ or $D^2_-\times S^2$, respectively, which explains
the Euler numbers~$0$ in Figure~\ref{figure:hopf-link}.
\end{exs}

If one wants to study the topology of $4$-manifolds
with the help of Kirby diagrams,
one needs to understand handle slides of $2$-handles. When the
attaching circle of one $2$-handle moves across another $2$-handle,
the attaching circle twists around that of the second handle
according to the Euler number of the second attaching circle,
and its Euler number (or, to use the correct term in this
more general context, its `framing') changes by that second Euler
number. You may try your hand with the following example, which is
the $4$-dimensional analogue of Proposition~\ref{prop:surfaces}.

\begin{prop}
\label{prop:dim4}
The $4$-manifolds $(S^2\times S^2)\#\overline{\CP}^2$ and
$\CP^2\#\overline{\CP}^2\#\overline{\CP}^2$ are diffeomorphic.
\end{prop}

For a comprehensive introduction to `Kirby calculus', the art of
manipulating Kirby diagrams, see~\cite{gost99}.
\section{Dimension five: open books}
\label{section:dim5}
An open book decomposition of an $n$-manifold $M$ is a way to write $M$
as a collection of $(n-1)$-dimensional diffeomorphic
submanifolds with boundary,
called the pages, which are bound together along their common boundary.
Take a look at Figure~\ref{figure:s3-book}, which is essentially
the Heegaard splitting of $S^3$ from Figure~\ref{figure:s3-heegaard1a},
except that the discs making up the solid torus $V_1$ in this splitting
have now been extended until they touch the soul $\{0\}\times S^1$
of the second solid torus $V_2\cong D^2\times S^1$. Notice that along that
circle, the discs come together like the pages of a book at the binding,
see Figure~\ref{figure:binding}.

\begin{figure}[h]
\labellist
\small\hair 2pt
\endlabellist
\centering
\includegraphics[scale=.4]{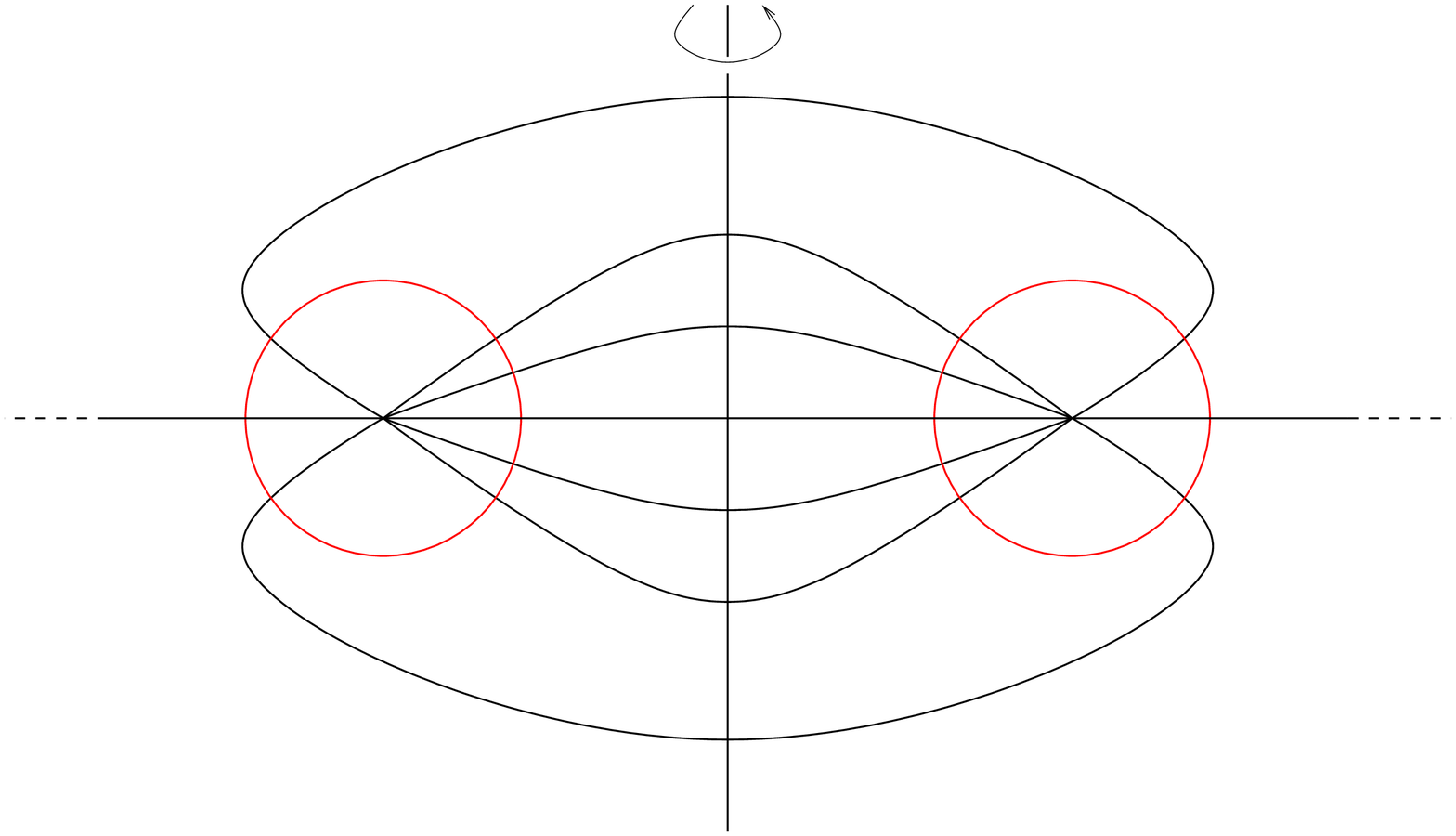}
  \caption{An open book decomposition of $S^3$.}
  \label{figure:s3-book}
\end{figure}

\begin{figure}[h]
\labellist
\small\hair 2pt
\pinlabel $S^1$ [r] at 237 105
\pinlabel $B$ [l] at 378 14
\pinlabel $\frakp^{-1}(\varphi)$ [l] at 386 435
\endlabellist
\centering
\includegraphics[scale=0.3]{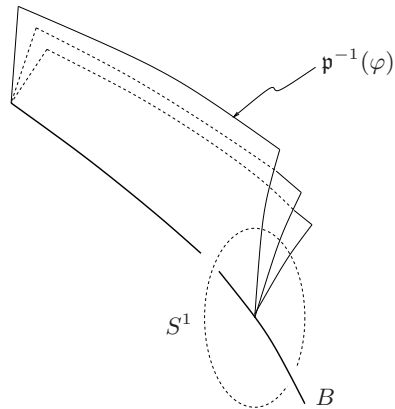}
  \caption{An open book near the binding.}
  \label{figure:binding}
\end{figure}

Here is one way to interpret this picture, giving rise to
the first of two equivalent definitions of an open book decomposition.

\begin{defn}
An \emph{open book decomposition} of a manifold $M$ consists of
a codimension~$2$ submanifold $B\subset M$, called the \emph{binding},
whose neighbourhood in $M$ is required to be of the form $B\times D^2$,
and a locally trivial fibration $\frakp\co M\setminus B\rightarrow S^1$,
which on $B\times\bigl(D^2\setminus\{0\}\bigr)$ is given by the
angular coordinate in the $D^2$-factor. The closure $\Sigma:=
\overline{\frakp^{-1}(\varphi)}$ of any fibre, $\varphi\in S^1$,
which is a codimension~$1$ submanifold of $M$ with boundary
$\partial\Sigma=B$, is called the \emph{page} of the open book.
\end{defn}

This definition makes perfect sense for topological
manifolds, but I shall now assume that we are working
in the differential category.
One may then choose a vector field on $M$ transverse to the pages,
equal to $\partial_{\varphi}$ near the binding, and such that it
projects onto the vector field $\partial_{\varphi}$ on $S^1$
under the differential of~$\frakp$. The time~$2\pi$ map of this
vector field then defines a diffeomorphism $\psi$ of~$\Sigma$,
equal to the identity near the boundary~$\partial\Sigma$.
This is called the \emph{monodromy} of the open book.

This leads to an alternative way of viewing an open book
decomposition. Start with a page $\Sigma$ and a
monodromy diffeomorphism $\psi\co \Sigma\rightarrow\Sigma$,
equal to the identity near~$\partial\Sigma$. Then form the
\emph{mapping torus}
\[ \Sigma\times [0,2\pi]/(x,2\pi)\sim(\psi(x),0).\]
By the assumption on $\psi$ this has boundary $\partial\Sigma\times S^1$,
and we obtain a closed manifold $M$ by gluing in a copy of
$\partial\Sigma\times D^2$. The binding will then be
$B=\partial\Sigma\times\{0\}$, and the projection $\frakp\co M\setminus B
\rightarrow S^1$ is given by the projection onto the second coordinate
in the mapping torus, and onto the angular coordinate
of the $D^2$-factor in $\partial\Sigma\times D^2$.
The page in the sense of the first definition will be $\Sigma$ with
a collar $\partial\Sigma\times [0,1]$ attached to its boundary,
where $1-r\in[0,1]$ corresponds to the radial coordinate $r$ in~$D^2$.

\begin{ex}
In the open book decomposition of $S^3$ in Figure~\ref{figure:s3-book}
we have binding $S^1$, page~$D^2$, and monodromy equal to
$\mathrm{id}_{D^2}$.
\end{ex}

See \cite{wink98} for a beautiful survey on the history of open books
and their applications to cobordism theory, foliations, differential
geometry etc. All odd-dimensional closed manifolds admit an
open book decomposition; manifolds of even dimension $\geq 6$
do so under some algebraic topological assumption, which
in the simply connected case is the vanishing of the signature.
Moreover, the pages in the decomposition of an $n$-dimensional
manifold may be assumed to have the homotopy type of a complex of
dimension $[n/2]$.

For a $5$-dimensional manifold this means that we can always find
an open book decomposition where the pages are $4$-dimensional
$2$-handlebodies. Thus, we have a way to describe the manifold
in terms of a Kirby diagram of the page, provided we can encode
the monodromy diffeomorphism in this picture. This is difficult,
in general, but even the identity map as monodromy
gives rise to many nontrivial examples.

\begin{exs}
(0) For any compact $4$-manifold $\Sigma$ with boundary
and $\psi=\id_{\Sigma}$ we have
\[ M=(\Sigma\times S^1)\cup_{\partial\Sigma\times S^1}
(\partial\Sigma\times D^2)=\partial (\Sigma\times D^2).\]

(1) The empty Kirby diagram can now be read as a picture for
$\partial (D^4\times D^2)= S^5$.

\vspace{1mm}

(2) The diagram in Figure~\ref{figure:s1d3} with a single $1$-handle
and $\psi=\mathrm{id}$ represents 
\[ \partial (S^1\times D^3\times D^2)=S^1\times S^4.\]

(3) Any $S^3$-bundle over $S^2=D^2_-\cup_{S^1}D^2_+$ splits into two trivial
bundles $D^2_-\times S^3$ and $D^2_+\times S^3$, and in the
bundle over $S^2$ the
$S^3$-fibres are glued over corresponding points of the equator
$D^2_-\cap D^2_+$ of the base sphere by an element of
the special orthogonal group
$\mathrm{SO}(4)$ varying continuously with the point on the
equator. It follows that these bundles are classified by the
fundamental group $\pi_1(\mathrm{SO}(4))=\Z_2$, so there is
only the trivial bundle $S^2\times S^3$ and a non-trivial one,
denoted by $S^2\,\tilde{\times}\,S^3$. The Euler number of
$D^2$-bundles over $S^2$ that we discussed earlier can be
interpreted as the corresponding element in $\pi_1(\mathrm{SO}(2))=\Z$.

It follows that
the diagram in Figure~\ref{figure:d2bundle1} with $\psi=\mathrm{id}$
represents
\begin{eqnarray*}
M & = & \partial\bigl( (\text{$D^2$-bundle over $S^2$})
        \times D^2\bigr)\\
  & = & \partial(\text{$D^4$-bundle over $S^2$})\\
  & = & \begin{cases}
        S^2\times S^3             & \text{for $m$ even},\\
        S^2\,\tilde{\times}\, S^3 & \text{for $m$ odd}.
        \end{cases}
\end{eqnarray*}
\end{exs}
\section{Contact structures on open books}
I now want to show how one can use contact geometry to
achieve a further dimensional reduction: $5$-manifolds
admitting a contact structure have open book
decompositions whose pages can be represented by $2$-dimensional
diagrams; if the monodromy can also be encoded
in the diagram, this gives a complete description of
the $5$-manifold. At least in principle this opens the possibility
that some $7$-dimensional contact manifolds can be represented by
$3$-dimensional diagrams.

It is not my intention to give an introduction to contact geometry;
for that I refer the reader to \cite{geig01} or~\cite{geig08}.
For a more detailed survey of the topics in this section
see~\cite{geig12}.
My sole aim here is to explain how contact geometry can be
used to simplify at least certain Kirby diagrams to drawings in
the $2$-plane.

\begin{defn}
The \emph{standard contact structure} on $\R^{2n+1}$
with cartesian coordinates $(x_1,y_1,\ldots,x_n,y_n,z)$ is
the hyperplane field $\xist=\ker(dz+\sum_i x_i\, \rmd y_i)$.
In other words, $\xist$ is spanned by the $2n$
vector fields $\partial_{x_i}, \partial_{y_i}-x_i\,\partial_z$, $i=1,
\ldots,n$.

A \emph{contact manifold} is an odd-dimensional manifold with
a tangent hyperplane field that looks locally like~$\xist$.
\end{defn}

The $3$-sphere $S^3\subset\C^2$ carries a natural contact structure, viz.\
the tangent field of real $2$-planes invariant under
the complex structure. The fact that this plane field is
a contact structures corresponds to what complex geometers
call the strict pseudoconvexity of $S^3$ in $\C^2$. On the
complement of a single point, this contact structure is
diffeomorphic to $(\R^3,\xist)$.

Among knots in $(\R^3,\xist)$ there is the distinguished class of
\emph{Legendrian knots} tangent to the plane field~$\xist$.
Consider a parametrised curve
\[ t\longmapsto\gamma(t)=\bigl(x(t),y(t),z(t)\bigr).\]
The condition for $\gamma$ to be Legendrian is $\dot{z}+x\dot{y}\equiv 0$.
Hence, if $\dot{y}(t)=0$ then $\dot{z}(t)=0$. This means that the
so-called \emph{front projection} $\gamma_{\mathrm{F}}(t)=
\bigl(y(t),z(t)\bigr)$ does not have any vertical tangencies,
but singular points instead. Away from
those singularities, the missing $x$-coordinate can be recovered from
$\gamma_{\mathrm{F}}$ via
\[ x(t)=-\frac{\dot{z}(t)}{\dot{y}(t)}=-\frac{\rmd z}{\rmd y},\]
that is, as the negative slope of the front projection in the
$yz$-plane. By a $C^2$-small perturbation of $\gamma$ one can achieve
that the front projection $\gamma_{\mathrm{F}}$ has only
isolated semi-cubical cusp singularities, i.e.\ around
a cusp at $t=0$ the curve $\gamma$ looks like
\[ \gamma(t)=(t,-t^2,2t^3/3),\]
see Figure~\ref{figure:cusp}.

\begin{figure}[h]
\labellist
\small\hair 2pt
\endlabellist
\centering
\includegraphics[scale=.35]{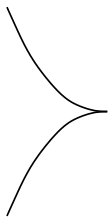}
  \caption{A cusp in the front projection of a Legendrian curve.}
  \label{figure:cusp}
\end{figure}

Notice that at a crossing of two curves in the front projection, the
strand with the smaller slope will be above that with the larger one.
This is illustrated in Figure~\ref{figure:LRlegtrefoil},
which shows front projections of Legendrian realisations
of the left-handed and right-handed trefoil knot
in Figure~\ref{figure:LRtrefoil}. In other words, it is superfluous
to indicate the under- and overcrossings in the
front projection, so we really have a strictly $2$-dimensional
picture of a knot in $3$-space.

\begin{figure}[h]
\labellist
\small\hair 2pt
\endlabellist
\centering
\includegraphics[scale=.4]{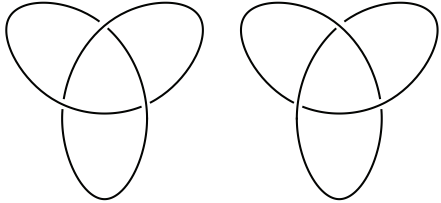}
  \caption{The left-handed and right-handed trefoil knot.}
  \label{figure:LRtrefoil}
\end{figure}

\begin{figure}[h]
\labellist
\small\hair 2pt
\endlabellist
\centering
\includegraphics[scale=.44]{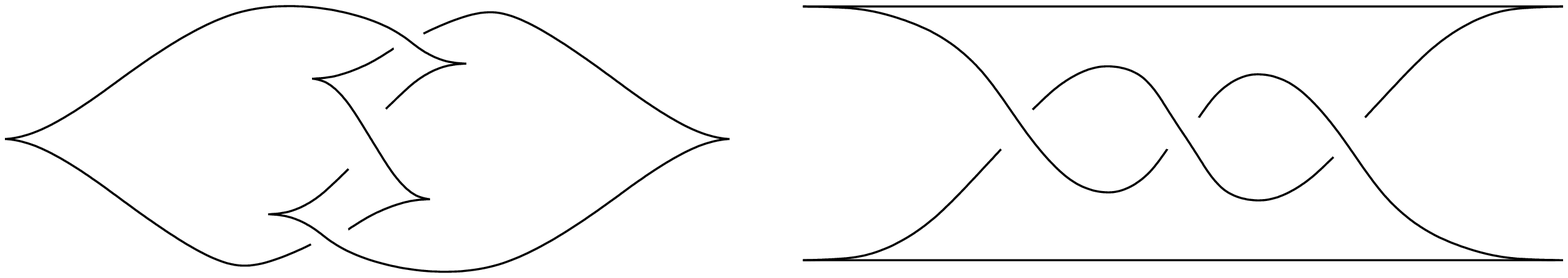}
  \caption{Front projections of Legendrian trefoils.}
  \label{figure:LRlegtrefoil}
\end{figure}

A Legendrian curve $\gamma$ in a $3$-dimensional
contact manifold has a natural
parallel curve, given by pushing $\gamma$ in a direction transverse
to the contact structure. For a Legendrian curve in
$(\R^3,\xist)$ this simply corresponds to
pushing its front projection in the $z$-direction. This parallel
curve (the `contact framing'), with one extra
left-handed twist, defines a preferred way of attaching
a $2$-handle to the $4$-ball, and it is the one for which
the complex structure on the $4$-ball extends to a Stein structure
on the $2$-handlebody.

Recall that a \emph{Stein manifold} is a complex manifold that admits
a proper holomorphic embedding into some complex affine space $\C^N$.
These admit exhausting Morse functions with strictly pseudoconvex
regular level sets. Sublevel sets of this Morse function are called
\emph{Stein domains}, and this is what I mean by a Stein structure on
a $2$-handlebody.

A closed contact manifold $(M,\xi)$ is called
\emph{Stein fillable} if there is a Stein domain $W$ with boundary
$\partial W=M$ such that the tangent hyperplane field
on $M$ invariant under the complex structure coincides with~$\xi$.
The homotopical dimension of a Stein domain is always at most half
of its topological dimension. If the homotopical dimension
is smaller than this maximum, the Stein domain is called
\emph{subcritical}.

By a deep result of Giroux~\cite{giro02}, there is an intimate
relation between contact structures and open books. In
dimension five this says that any
contact manifold has an open book decomposition where the pages
are $4$-dimensional Stein domains, i.e.\
$2$-handlebodies with the $2$-handles
attached along Legendrian knots, using the preferred framing.
Thus, contact geometry leads to a further simplification:
the attaching of a $2$-handle is encoded in the
Legendrian knot itself, so we no longer need to keep
track of framings under handle slides.

A theory of diagrams for contact $5$-manifolds has been developed
in~\cite{dgk12}. Here are some sample applications of this theory.

\begin{app}
(1) One can prove equivalences between contact manifolds.
For instance, the manifolds $(S^2\times S^3)\#(S^2\,\tilde{\times}\, S^3)$
and $(S^2\,\tilde{\times}\, S^3)\#(S^2\,\tilde{\times}\, S^3)$
(with appropriate contact structures) are diffeomorphic.
Even for the purely topological statement, the handle moves simplify by
taking the contact geometric viewpoint.

\vspace{1mm}

(2) The diagrams can be used for a partial classification of
subcritically Stein fillable contact $5$-manifolds.
For a further discussion of this issue see~\cite{dgz17}.

\vspace{1mm}

(3) Contact $5$-manifolds resulting from open books
with homeomorphic but non-diffeomorphic pages
have been studied in \cite{ozko15} and~\cite{akya15}.

\vspace{1mm}

(4) One can give a diagrammatic proof of the result, first shown
in~\cite{geig91}, that every simply connected $5$-manifold
admits a contact structure, provided a certain obvious
topological condition (reduction of the structure group of the
tangent bundle to the unitary group~$\mathrm{U}(2)$) is satisfied.

\vspace{1mm}

(5) A purely topological application is the following.
Every simply connected $5$-dimensional spin manifold (that is, a manifold
with vanishing second Stiefel--Whitney class) is a double branched cover
of~$S^5$. This result depends in an essential way on open
books with non-trivial monodromy. When a $2$-handle is attached
along an unknot, one can define a Dehn twist along the
resulting $2$-sphere in the handlebody. It turns out that all
the $5$-manifolds in question have an open book decomposition whose
monodromy is a square $\psi^2$ of a diffeomorphism $\psi$
composed of such Dehn twists, and that the open book
with the same page but monodromy $\psi$ is diffeomorphic to~$S^5$.
This immediately gives the claimed branched cover description,
with the binding as the branching set.
\end{app}
\begin{ack}
I am grateful to Peter Albers and Marc Kegel for their comments on
a draft version of this paper. Special thanks to Guido Sweers
for creating Figures \ref{figure:rp2}, \ref{figure:3-dim-1-handle}
and~\ref{figure:LRtrefoil}.
The research of the author is supported by the SFB/TRR 191
`Symplectic Structures in Geometry, Algebra and Dynamics', funded by the
Deutsche Forschungsgemeinschaft.
\end{ack}

\end{document}